\documentclass[review,hidelinks,onefignum,onetabnum]{siamart250211}
\preto{\nolinenumbers}

\usepackage{lipsum}
\usepackage{amsfonts}
\usepackage{graphicx}
\usepackage{epstopdf}
\usepackage{algorithmic}
\ifpdf
  \DeclareGraphicsExtensions{.eps,.pdf,.png,.jpg}
\else
  \DeclareGraphicsExtensions{.eps}
\fi

\newsiamremark{remark}{Remark}
\newsiamremark{hypothesis}{Hypothesis}
\crefname{hypothesis}{Hypothesis}{Hypotheses}
\newsiamthm{claim}{Claim}
\newsiamremark{fact}{Fact}
\crefname{fact}{Fact}{Facts}

\headers{Quotient Manifold Optimization for Spectral Compressed Sensing}{Wenlong Wang, Wen Huang, and Zai Yang}

\title{Quotient Manifold Optimization for Spectral Compressed Sensing\thanks{Submitted to the editors DATE.
\funding{Zai Yang was supported by the National Natural Science Foundation of China under Grant 12371464. Wen Huang was partially supported by the National Natural Science Foundation of China (No. 12371311), the Natural Science Foundation of Fujian Province (No. 2023J06004), and the Fundamental Research Funds for the Central Universities (No. 20720240151). (Corresponding author: Zai Yang)}
}}

\author{Wenlong Wang\thanks{School of Mathematics and Statistics, Xi'an Jiaotong University, Xi'an 710049, China
  (\email{wang1813857265@stu.xjtu.edu.cn}, \email{yangzai@xjtu.edu.cn}).}
\and Wen Huang\thanks{School of Mathematical Sciences, Xiamen University, Xiamen 361005, China 
  (\email{wen.huang@xmu.edu.cn}).}
\and Zai Yang\footnotemark[2]}

\usepackage{amsopn}
\DeclareMathOperator{\diag}{diag}

\usepackage{mathrsfs}

\def\st{\text{subject to }}

\def\m #1{\boldsymbol{#1}}

\def\bee{\begin{equation}}
\def\ene{\end{equation}}

\def\beq{\begin{eqnarray}}
\def\enq{\end{eqnarray}}

\newtheorem{thm}{Theorem}
\newtheorem{prop}{Proposition}

\def\diag #1{\text{diag}#1}
\def\tr #1{\text{tr}#1}

\def\rank #1{\text{rank}#1}
\def\st {\text{ subject to }}

\ifpdf
\hypersetup{
  pdftitle={Quotient Manifold Optimization for Spectral Compressed Sensing},
  pdfauthor={Wenlong Wang, Wen Huang, and Zai Yang}
}
\fi
\begin{document}
\nolinenumbers
\maketitle
\begin{abstract}Spectral compressed sensing involves reconstructing a spectral-sparse signal from a subset of uniformly spaced samples, with applications in radar imaging and wireless channel estimation. By fully exploiting the signal structures, this problem is formulated as a rank-constrained semidefinite program subject to Hankel-Toeplitz structural constraints in our previous work. To further enhance computational efficiency, this paper proposes a quotient-manifold-based optimization framework that leverages the underlying Riemannian geometry in a matrix factorization space. Specifically, we establish an equivalence between spectral-sparse signals and matrix equivalence classes under the action of the real orthogonal group, where each class member corresponds to a rank-constrained positive-semidefinite Hankel-Toeplitz structured matrix. The associated quotient manifold geometry—including the Riemannian metric, horizontal space, retraction, and vector transport—is rigorously derived. Based on these results, we develop a Riemannian conjugate gradient descent algorithm, where each iteration is efficiently implemented using fast Fourier transforms (FFTs) by exploiting the Hankel and Toeplitz structures. Extensive numerical experiments demonstrate the superior performance of the proposed algorithm in both computational speed and accuracy compared to state-of-the-art methods.

\end{abstract}



\section{Introduction}
\label{Intro}
Spectral compressed sensing aims to recover a spectral-sparse signal and estimate its spectral components (a.k.a., the continuous-valued frequencies and their complex-valued amplitudes) from a subset of its uniformly spaced samples~\cite{stoica2005spectral,duarte2013spectral,candes2014towards}. It is an extension of the prominent compressed sensing problem by defining the sparse signal on a continuous interval and modifying the sensing matrix as a continuous-time Fourier transform with sampling. In the case of full uniform samples, it is usually known as line spectral estimation~\cite{stoica2005spectral}, harmonic retrieval~\cite{gorodnitsky1997sparse}, and spectral super-resolution~\cite{polisano2019convex,chen2001atomic,candes2014towards}. In this paper, we emphasize the challenging scenario of partial samples or compressive measurements that arises in various applications such as array signal processing~\cite{yang2018sparse,10144727}, radar imaging~\cite{cuyt2020sparse}, wireless channel estimation~\cite{tsai2018millimeter}, and health monitoring~\cite{wilber2022data,duan2014compressed}.

The research of spectral compressed sensing has a long history, with a primary focus on designing advanced algorithms to enhance accuracy and computational efficiency. A central challenge lies in the nonlinear relationship between the observed data and the underlying frequencies. Prony's method, dating to 1795, establishes a connection between the frequency parameters and the roots of a polynomial, solving for them by constructing a system of linear equations~\cite{osborne1995modified,de1795essai}. The discrete Fourier transform (DFT) converts a finite sequence of uniformly spaced time-domain samples into the frequency domain to extract the dominant spectral components~\cite{sundararajan2001discrete}, and its fast realization, the fast Fourier transform (FFT), is honored as one of the top 10 algorithms of the 20th century~\cite{dongarra2000guest}. Subspace-based methods, such as multiple signal classification (MUSIC)~\cite{stoica1989music} and estimation of signal parameters via rotational invariant techniques (ESPRIT)~\cite{roy1989esprit}, estimate the frequencies using the signal or noise subspace obtained from the sample covariance matrix, given prior knowledge on the number of spectral components. However, the performance of these methods deteriorates in the case of compressive samples~\cite{pillai1989forward}.

Following the seminal works of compressed sensing~\cite{donoho2006compressed}, sparse optimization methods have been extensively studied for the deeply connected topic of spectral compressed sensing, which resolves naturally the issue of compressive measurements in an optimization framework. In fact, the work of Candes, Romberg and Tao~\cite{candes2006robust} studies exactly the problem of spectral compressed sensing, which can be implied by the title of~\cite{candes2006robust}, but with discrete-valued frequencies (as opposed to the continuous setting of this paper). The discretization causes a fundamental off-grid problem between the true frequencies and the discretization grid and has raised major concerns regarding algorithm design and theoretical analysis~\cite{yang2018sparse,yang2012off,hu2013fast,tan2014joint}. While various off-grid sparse methods have been proposed to mitigate the effect of the off-grid problem, it remains unresolved until the seminal works of Candes  and Recht and their collaborators~\cite{candes2013super,candes2014towards,tang2013compressed,bhaskar2013atomic}, leading to gridless sparse methods. In these papers, the widely used $\ell_1$ norm optimization in compressed sensing is extended successfully to the continuous-frequency setting and then computed exactly (without discretization) using semidefinite programming (SDP). In particular, the continuous analog of  $\ell_1$ norm is called the atomic norm or the total variation norm. It is computed, by applying the Carath\'eodory-Fej\'er Theorem on positive-semidefinite (PSD) Toeplitz matrices~\cite{caratheodory1911zusammenhang}, as the trace norm of a PSD block matrix consisting of a Toeplitz matrix to optimize. Moreover, a continuous analog of the $\ell_0$ pseudo-norm, called the atomic $\ell_0$ norm, is also developed as the rank of the aforementioned PSD block matrix, which will be referred to as the low-rank PSD Toeplitz model hereafter. As in compressed sensing, the atomic norm is a convex relaxation of the atomic $\ell_0$ norm.

The convex atomic norm approaches offer rigorous performance guarantees; however, they still suffer from major limitations. First, these convex relaxation methods require the spectral components to be sufficiently separated for successful recovery, which inevitably imposes a resolution limit~\cite{candes2014towards,tang2015resolution,yang2024separation,candes2013super}. Second, solving the SDPs arising from atomic norm formulations entails high computational complexity, limiting the practical deployment of these methods in real-time applications such as radar and wireless communication systems. To address the resolution limitation, a reweighted atomic norm method is studied in~\cite{yang2015enhancing} which however further increases the computational burden. To reduce the computational complexity, efficient algorithms such as those based on the alternating direction method of multipliers (ADMM) have been developed~\cite{10510492,bhaskar2013atomic,7080862}. Nevertheless, solving SDPs remains computationally demanding due to large-scale matrix multiplications and eigenvalue decompositions (EDs). This paper aims to propose one algorithm that resolves these limitations simultaneously.

In this paper, building on the authors' previous work~\cite{wu2022maximum}, we formulate the reconstruction of a spectral-sparse signal as a rank-constrained SDP with Hankel-Toeplitz structural constraints which fully exploits the signal structures. Inspired by recent advances in low-rank matrix recovery, we propose a Riemannian optimization algorithm to solve this problem efficiently. Our main contributions are summarized as follows.
\begin{itemize}
\item We show that a spectral-sparse signal uniquely corresponds to an equivalence class under the action of real orthogonal group, where each matrix in the class is associated with a rank-constrained PSD Hankel-Toeplitz matrix. We prove that the space of these equivalence classes forms a quotient manifold, thereby reformulating spectral compressed sensing as an optimization problem on this manifold. We establish the geometric structures required for Riemannian optimization, including the Riemannian metric, horizontal space, retraction, and vector transport.
\item We develop a Riemannian conjugate gradient descent algorithm for the quotient manifold optimization problem, referred to as the Hankel-Toeplitz Riemannian Conjugate Gradient Descent (HT-RCGD) algorithm. We present an efficient implementation of HT-RCGD by using FFTs and analyze its computational complexity, demonstrating its low per-iteration cost. Moreover, we provide a convergence guarantee for the algorithm.
\end{itemize}
We conduct extensive numerical experiments to evaluate the performance of HT-RCGD. The results demonstrate its superiority over existing methods in terms of convergence rate and phase transition behavior, confirming its effectiveness for spectral compressed sensing.

\subsection{Related work}
Spectral compressed sensing is formulated in~\cite{andersson2014new,holland2011fast,cai2023structured} as a low-rank Hankel matrix recovery problem when the undamped nature of the signal is disregarded, leading to several nonconvex algorithms such as Hankel-based PGD (H-PGD)~\cite{cai2018spectral}, symmetric Hankel PGD (SH-PGD)~\cite{li2024projected}, and Hankel-based low-rank projected proximal gradient (H-LPPG)~\cite{yao2025low}. However, it is shown in~\cite{yang2024new} that the low-rank Hankel model introduces new damping factors and thus cannot utilize the full signal structures. This relaxation leads to suboptimal parameter identifiability and higher estimation error~\cite{wu2022maximum}. While a low-rank double-Hankel model is proposed in~\cite{yang2024new} to alleviate this issue, it remains inexact and may limit achievable estimation accuracy. In contrast, this paper adopts the low-rank PSD Hankel-Toeplitz model, which provides an exact characterization of spectral sparsity. Inspired by the advances in fast algorithms for low-rank Hankel recovery, the proposed HT-RCGD algorithm leverages FFTs to exploit the circulant structures of Hankel and Toeplitz matrices for efficient computation. Moreover, the recent success in manifold optimization for Hankel recovery~\cite{bian2024preconditioned-SAM} motivates our integration of quotient geometry to enhance Hankel-Toeplitz optimization.

Several studies have been attempted to solve Toeplitz-based optimization models for spectral compressed sensing. However, the low-rank Toeplitz representation in~\cite{tang2013compressed,7451201}, whose convex relaxation is linked to atomic norm minimization, has been shown to be unsuitable for nonconvex optimization, as it is neither unique nor bounded, potentially causing convergence issues~\cite{wu2022maximum}. In contrast, the PSD Hankel-Toeplitz representation used in this paper offers a unique and exact characterization of spectral-sparse signals. A matrix factorization approach for this model is proposed in~\cite{wu2024fast}, where the constrained optimization problem is reformulated into an equivalent unconstrained problem and solved using projected gradient descent, referred to as Hankel-Toeplitz projected gradient descent (HT-PGD). In HT-PGD, the projection step ensures the iterates remain within the solution set to guarantee convergence. This paper improves upon HT-PGD by leveraging quotient geometry to eliminate factorization redundancy and employing a Riemannian conjugate gradient descent algorithm to enhance iteration efficiency. Additionally, a regularization term is introduced into the objective function to ensure convergence, thereby avoiding the need for explicit projection.

Riemannian manifold optimization has been employed to study the general low-rank matrix recovery problem~\cite{vandereycken2013low, bioli2025preconditioned, absil2008optimization, bian2024preconditioned, zheng2022riemannian, wei2016guarantees}. These techniques represent low-rank matrices using compact variables through quotient geometry~\cite{huang2017solving, mishra2014riemannian} or embedded geometry~\cite{vandereycken2013low, wei2016guarantees}, addressing the redundancy and non-uniqueness inherent in matrix factorization representations used in Euclidean optimization methods~\cite{tanner2016low,sun2016guaranteed}. By leveraging the geometric structure of low-rank or low-rank PSD matrix manifolds, advanced algorithms such as Riemannian conjugate gradient descent have been developed, reducing per-iteration computational complexity and demonstrating improved convergence behavior~\cite{zheng2022riemannian, wei2016guarantees}. In this paper, we observe that the factorization of a rank-constrained PSD Hankel-Toeplitz matrix induces an equivalence class under the action of real orthogonal group, and the space of these equivalence classes forms a quotient manifold. Building on this insight, we employ an advanced Riemannian conjugate gradient descent algorithm to efficiently solve the corresponding quotient manifold optimization problem.

\subsection{Notations and organization}

Notations used in this paper are as follows. Bold face letters denote vectors and matrices. The sets of real and complex numbers are represented by $\mathbb{R}$ and $\mathbb{C}$, respectively. Given a vector $\m{x}$, its transpose, complex conjugate, conjugate transpose, and $l_2$ norm are denoted as $\m{x}^\top$, $\overline{\m{x}}$, $\m{x}^H$ and $||\m{x}||_2$, respectively. For matrix $\m{X}$, $\m{X}^\top$, $\overline{\m{X}}$, $\m{X}^H$, $\m{X}^{-1}$, $\m{X}^{\dag}$, $\left\|\m{X}\right\|_\text{F}$, $\rank\left(\m{X}\right)$ and $\tr\left(\m{X}\right)$ denote its transpose, complex conjugate, conjugate transpose, inverse, Moore-Penrose pseudo inverse, Frobenius norm, rank, and trace, respectively. The real part of a number or a matrix is denoted by $\Re\left\{\cdot\right\}$. The real inner product of matrices $\m{X}$ and $\m{Y}$ is defined as $\left\langle\m{X},\m{Y}\right\rangle=\Re\left\{\tr(\m{X}^H\m{Y})\right\}$. For an integar $N>0$, the notation $1:N$ represents the set $\left\{1,2,\dots,N\right\}$. The $j$-th entry of vector $\m{x}$ is $x_j$, and the $\left(i_1,i_2\right)$ entry of matrix $\m{X}$ is $X_{i_1,i_2}$. The $i$-th row and column of a matrix $\m{X}$ are denoted as $\m{X}_{i,:}$ and $\m{X}_{:,i}$, respectively. $\m{X}\geq\m{0}$ means that $\m{X}$ is PSD. $\m{X}>\m{0}$ means that $\m{X}$ is Hermitian positive-definite. The diagonal matrix with vector $\m{x}$ on its diagonal is represented as $\diag\left(\m{x}\right)$. The notation $|\cdot|$ denotes the absolute value of a scalar, the determinant of a matrix, or the cardinality of a set. The identity matrix of size $N$ is denoted as $\m{I}_N$. $\mathcal{I}$ denotes the identity operator. The adjoint of a linear operator $\mathcal{A}$ is denoted by $\mathcal{A}^*$. For a Riemannian manifold $\mathcal{M}$, the tangent space at $\m{p}\in\mathcal{M}$ is denoted as $\text{T}_{\m{p}}\mathcal{M}$. $\mathcal{O}^K$ denotes the real orthogonal group $\left\{\m{O}\in\mathbb{R}^{K\times K}:\;\m{O}\m{O}^\top=\m{O}^\top\m{O}=\m{I}_K\right\}$. The symbol $\oplus$ denotes the direct sum of linear spaces. The notation $\mathbb{C}^{p\times K}_*$ denotes the set $\left\{\m{Z}\in\mathbb{C}^{p\times K}:\;\rank\left(\m{Z}\right)=K\right\}$.

The remainder of this paper is organized as follows. Section \ref{Sec:Problem Formulation} introduces the Hankel-Toeplitz matrix optimization model for spectral compressed sensing. Section \ref{Sec:Quotient Manifold} reformulates the problem as an optimization over equivalence classes on a quotient manifold and establishes the corresponding geometric structures. Section \ref{Sec:HT-RCGD} presents the HT-RCGD algorithm in detail. Section \ref{Sec:simulation} presents numerical simulations, and Section \ref{Sec:conclusion} concludes the paper.

\section{Preliminaries}
\label{Sec:Problem Formulation}
In this section, we introduce the spectral compressed sensing problem and present the low-rank Hankel-Toeplitz model, which serves as the foundation of the proposed optimization problem and algorithm.

A spectral-sparse signal is modeled as
\begin{equation}
\begin{aligned}
y_{n}=\sum_{k=1}^Ks_ke^{j2\pi \left(n-1\right){f}_k},\;n=1,\cdots,N,
\end{aligned}
\end{equation}
where $N$ is the number of samples, $s_k$ is the coefficient of the $k$-th component, and $f_k$ denotes its frequency. Spectral compressed sensing aims to recover the frequency vector $\m{f}=\left[f_1,\dots,f_K\right]^\top$, the coefficient vector $\m{s}=\left[s_1,\dots,s_K\right]^\top$, and the complete signal $\m{y}=\left[y_1,\dots,y_N\right]^\top$ from the partial observations $\left\{y_n\right\}_{n\in\Omega}$, where $\Omega \subseteq \left\{1,\dots,N\right\}$.

The recovery of a spectral-sparse signal can be formulated as the following feasibility problem:
\begin{equation}
\label{feasibility-problem}
\begin{aligned}
\text{find}\:\m{y}\in\mathcal{S}_0,\st\text{$\mathcal{P}_\Omega\left(\m{y}\right)$ is given},
\end{aligned}
\end{equation}
where $\mathcal{S}_0$ denotes the set of spectral-sparse signals with $K$ distinct spectral components:
\begin{equation}
\label{spectrally-sparse-signal-set}
\begin{aligned}
\mathcal{S}_0=\Bigg\{\m{x}\in\mathbb{C}^N:\;&x_n=\sum_{k=1}^Ks_ke^{j2\pi \left(n-1\right)f_k},\;n=1,\cdots,N,\\
&s_{k}\neq0,k=1,\cdots,K,\\
&f_{k_1}\neq f_{k_2},\; k_1,k_2=1,\cdots,K,\;k_1\neq k_2\Bigg\},
\end{aligned}
\end{equation}
$\mathcal{P}_\Omega$ is the observation operator that preserves only the entries of $\m{y}$ indexed by $\Omega$, and $\left|\Omega\right|$ denotes its cardinality. To facilitate the ensuing construction of structured matrices, we assume $N$ is odd and define $p=\frac{N+1}{2}$. If $N$ is even, the $(N+1)$-th sample can be treated as missing. Notably, once the original signal $\m{y}$ is exactly reconstructed, the frequencies $\m{f}$ and coefficients $\m{s}$ can be determined efficiently, provided that $K<\left(N+1\right)/2$, see~\cite{yang2018sparse}. Since a feasible solution exists, problem~\eqref{feasibility-problem} can be reformulated as the following optimization problem:
\begin{equation}
\label{optimization-problem}
\begin{aligned}
\min_{\m{x}}\;\rho\left(\mathcal{P}_\Omega\left(\m{x}\right),\mathcal{P}_\Omega\left(\m{y}\right)\right),\st\m{x}\in\mathcal{S}_0,
\end{aligned}
\end{equation}
where $\rho\left(\cdot,\cdot\right)$ is a specified distance function satisfying $\rho\left(\mathcal{P}_\Omega\left(\m{x}\right),\mathcal{P}_\Omega\left(\m{y}\right)\right)>0$ if $\mathcal{P}_\Omega\left(\m{x}\right)\neq\mathcal{P}_\Omega\left(\m{y}\right)$. Solving~\eqref{optimization-problem} is challenging due to the nonlinear dependence of $\m{y}$ on the frequency vector $\m{f}$ in the parameterization of $\mathcal{S}_0$. To address this issue, various matrix optimization models have been developed to reformulate~\eqref{optimization-problem} into a more tractable form by embedding the spectral-sparse signal into a low-rank structured matrix~\cite{andersson2014new,holland2011fast,tang2013compressed,candes2006robust,wu2022maximum}.

It is shown in \cite[Theorem 1]{wu2022maximum} that if $K<p$, then the set $\mathcal{S}_0$ can be equivalently reparameterized using a rank-constrained PSD Hankel-Toeplitz matrix as
\begin{equation}
\label{set-HT}
\begin{aligned}
\mathcal{S}_{0}=\mathcal{S}_{\text{HT}}=\left\{\m{x}\in\mathbb{C}^N:\;\begin{bmatrix}\mathcal{T}\overline{\m{t}}&\mathcal{H}\overline{\m{x}}\\\mathcal{H}\m{x}&\mathcal{T}\m{t}\end{bmatrix}\in\mathcal{S}_+^K,\;\text{for some $\m{t}\in\mathbb{C}^N$}\right\},
\end{aligned}
\end{equation}
where $\mathcal{H}$ is a Hankel operator that maps $\m{x}\in\mathbb{C}^{2p-1}$ to a $p\times p$ Hankel matrix $\mathcal{H}\m{x}$ with its $\left(i_1,i_2\right)$ entry given by $x_{i_1+i_2-1}$, $\mathcal{T}$ is a Toeplitz operator that maps $\m{t}\in\mathbb{C}^{2p-1}$ to a $p\times p$ Toeplitz matrix $\mathcal{T}\m{t}$ with its $\left(i_1,i_2\right)$ entry given by $t_{i_1-i_2+p}$, and $\mathcal{S}_+^K$ denotes the set of PSD matrices of fixed rank $K$. Specifically, for any spectral-sparse signal $\m{x}\in\mathcal{S}_0$ with frequency vector $\m{f}$ and cofficient vector $\m{s}$, there exists a unique $\m{t}$ in~\eqref{set-HT} given by $\mathcal{T}\m{t}=\m{A}\left(\m{f}\right)\diag(\left|\m{s}\right|)\m{A}^H\left(\m{f}\right)$~\cite{wu2022maximum}, where $\left|\cdot\right|$ is applied element-wise to $\m{s}$ and
\begin{equation}
\begin{aligned}
\m{A}\left(\m{f}\right)=\begin{bmatrix}1&1&\cdots&1\\
e^{j2\pi f_1}&e^{j2\pi f_2}&\cdots&e^{j2\pi f_K}\\
\vdots&\vdots&\vdots&\vdots\\
e^{j2\pi f_1\left(p-1\right)}&e^{j2\pi f_2\left(p-1\right)}&\cdots&e^{j2\pi f_K\left(p-1\right)}\end{bmatrix}.
\end{aligned}
\end{equation}
Using the equivalence between spectral-sparse signals and low-rank PSD Hankel-Toeplitz matrices in~\eqref{set-HT}, we obtain the following reformulation:
\begin{equation}
\label{optimization-problem-HT}
\begin{aligned}
\min_{\m{x}\in\mathbb{C}^N,\m{t}\in\mathbb{C}^N}\;\rho\left(\mathcal{P}_\Omega\left(\m{x}\right),\mathcal{P}_\Omega\left(\m{y}\right)\right),\st\begin{bmatrix}\mathcal{T}\overline{\m{t}}&\mathcal{H}\overline{\m{x}}\\\mathcal{H}\m{x}&\mathcal{T}\m{t}\end{bmatrix}\in\mathcal{S}_+^K.
\end{aligned}
\end{equation}

An equivalent reformulation of~\eqref{optimization-problem-HT} as a Hankel-Toeplitz matrix factorization problem is presented in~\cite{wu2024fast}. This formulation leverages a matrix factorization $\m{Z}$ to exploit the low-rank PSD structure as
\begin{equation}
\label{optimzation-problem-HT-factorization}
\begin{aligned}
\begin{bmatrix}\mathcal{T}\overline{\m{t}}&\mathcal{H}\overline{\m{x}}\\\mathcal{H}\m{x}&\mathcal{T}\m{t}\end{bmatrix}=\begin{bmatrix}\overline{\m{Z}}\\\m{Z}\end{bmatrix}\begin{bmatrix}\overline{\m{Z}}\\\m{Z}\end{bmatrix}^H=\begin{bmatrix}\overline{\m{Z}}{\overline{\m{Z}}}^H&\overline{\m{Z}}\m{Z}^H\\\m{Z}{\overline{\m{Z}}}^H&\m{Z}\m{Z}^H\end{bmatrix},
\end{aligned}
\end{equation}
and expresses the signal $\m{x}$ in terms of $\m{Z}$ as
\begin{equation}
\label{optimal-x}
\begin{aligned}
\m{x}=\left(\mathcal{H}^*\mathcal{H}\right)^{-1}\mathcal{H}^*\left(\m{Z}\m{Z}^\top\right)=\m{D}^{-2}\mathcal{H}^*\left(\m{Z}\m{Z}^\top\right),
\end{aligned}
\end{equation}
where $\m{D}=\diag\left(\left[1,\sqrt{2},\cdots,\sqrt{p},\cdots,\sqrt{2},1\right]^\top\right)$, and the last equality holds since $\m{D}=\left(\mathcal{H}^*\mathcal{H}\right)^{\frac{1}{2}}$ when $N$ is odd. Therefore, the optimization problem~\eqref{optimization-problem-HT} reduces to the following formulation~\cite{wu2024fast}:
\begin{equation}
\label{optimzation-problem-HT-factorization-Z}
\begin{aligned}
\min_{\m{Z}\in\mathbb{C}^{p\times K}_*}\;&\rho\left(\mathcal{P}_\Omega\left(\m{D}^{-2}\mathcal{H}^*\left(\m{Z}\m{Z}^\top\right)\right),\mathcal{P}_\Omega\left(\m{y}\right)\right),\\
\st&\left(\mathcal{I}-\mathcal{G}_1\right)\left(\m{Z}\m{Z}^\top\right)=\m{0},\;\left(\mathcal{I}-\mathcal{G}_2\right)\left(\m{Z}\m{Z}^H\right)=\m{0},
\end{aligned}
\end{equation}
where $\mathcal{G}_1=\mathcal{H}\left(\mathcal{H}^*\mathcal{H}\right)^{-1}\mathcal{H}^*$ and $\mathcal{G}_2=\mathcal{T}\left(\mathcal{T}^*\mathcal{T}\right)^{-1}\mathcal{T}^*$ denote the orthogonal projections onto the Hankel and Toeplitz subspaces, respectively.

\section{Optimization Problem on Quotient Manifold}
\label{Sec:Quotient Manifold}

In this section, we reformulate the low-rank Hankel-Toeplitz matrix optimization problem for spectral compressed sensing as an optimization over equivalence classes. We show that the set of all such equivalence classes constitutes a quotient manifold, and subsequently establish its Riemannian geometry for use in algorithm design.  The main challenges are twofold: 1) constructing the quotient manifold structure to capture the non-uniqueness of matrix factorizations through equivalence classes under group action, which requires verifying the properties of the associated group action, a form that has not been previously studied; and 2) carrying out rigorous derivations to establish the Riemannian geometry, including the metric, horizontal space, retraction, and vector transport.


\subsection{Problem formulation}
\begin{sloppypar}
Our reformulation is motivated by the observation that the optimal solution $\m{Z}^*$ to the matrix factorization problem in~\eqref{optimzation-problem-HT-factorization-Z} is non-unique. To be specific, both $\m{Z}^*{\m{Z}^*}^\top$ and $\m{Z}^*{\m{Z}^*}^H$ remain invariant under the action of real orthogonal group. We have the following result.
\begin{thm}
\label{prop-x2Z}
Assume $K<p$. For any spectral-sparse signal $\m{x}\in\mathcal{S}_0$ with frequency vector $\m{f}$ and cofficient vector $\m{s}$, there exists a unique equivalence class $\left[\m{Z}\right]=\left\{\m{Z}\m{O}\in\mathbb{C}^{p\times K}_*:\;\m{O}\in\mathcal{O}^K\right\}$ such that for every $\m{Z}\in\left[\m{Z}\right]$ it holds true that $\mathcal{H}\m{x}=\m{Z}\m{Z}^\top$ and $\m{Z}\m{Z}^H$ is a Toeplitz matrix.
\end{thm}
\begin{proof}
It is shown in \cite[Theorem 2.1]{wu2024fast} that
\begin{equation}
\label{set-reparameterize}
\begin{aligned}
\mathcal{S}_0=\left\{\m{x}\in\mathbb{C}^N:\;\mathcal{H}\m{x}=\m{Z}\m{Z}^\top,\mathcal{T}\m{t}=\m{Z}\m{Z}^H\text{ for some $\m{t}\in\mathbb{C}^N$ and $\m{Z}\in\mathbb{C}^{p\times K}_*$}\right\}.
\end{aligned}
\end{equation}
Furthermore, the set of valid $\m{Z}$ satisfying the constraints in~\eqref{set-reparameterize} is given by $\left\{\m{A}\left(\m{f}\right)\left(\diag\left(\m{s}\right)\right)^{\frac{1}{2}}\m{O}:\;\m{O}\in\mathcal{O}^K\right\}$. 
It is shown in \cite[Theorem 1]{wax1989unique} that the pair $\left\{\m{f},\m{s}\right\}$ is uniquely determined by $\m{x}$ when $K<p$. Consequently, we obtain the corresponding unique equivalence class $\left[\m{A}\left(\m{f}\right)\left(\diag(\m{s})\right)^{1/2}\right]$, thereby completing the proof.
\end{proof}
\end{sloppypar}

We next show that the set of all such equivalence classes forms a quotient manifold, as stated in the following theorem.
\begin{thm}
\label{thm-quotient}
The set 
\begin{equation}
\begin{aligned}
\left\{\left[\m{Z}\right]=\left\{\m{Z}\m{O}\in\mathbb{C}^{p\times K}_*:\;\m{O}\in\mathcal{O}^K\right\}:\m{Z}\in\mathbb{C}^{p\times K}_*\right\}=\mathbb{C}_*^{p\times K} / \mathcal{O}^K
\end{aligned}
\end{equation}
is a quotient manifold of dimension $2pK - \frac{K(K-1)}{2}$.
\end{thm}
\begin{proof}
Consider the set of equivalence classes, which corresponds exactly to the orbits of the real orthogonal group action $\mathcal{O}^K$ on $\mathbb{C}_*^{p\times K}$. It is straightforward to verify that $\mathcal{O}^K$ acts smoothly on $\mathbb{C}_*^{p\times K}$, thereby satisfying the properness according to \cite[Corollary 21.6]{lee2012smooth}. Moreover, for any $\m{Z} \in \mathbb{C}_*^{p\times K}$ and $\m{O} \in \mathcal{O}^K$, the equality $\m{Z}\m{O} = \m{Z}$ holds if and only if $\m{O} = \m{I}_K$, confirming that the group action is free. Therefore, by \cite[Theorem 21.10]{lee2012smooth}, the quotient space $\mathbb{C}_*^{p\times K} / \mathcal{O}^K$ is a quotient manifold. The dimension of $\mathbb{C}_*^{p\times K}$ is $2pK$, and the dimension of $\mathcal{O}^K$ is $\frac{K(K-1)}{2}$, indicating that the dimension of the quotient manifold $\mathbb{C}_*^{p\times K} / \mathcal{O}^K$ is $2pK - \frac{K(K-1)}{2}$.
\end{proof}

By leveraging Theorem \ref{prop-x2Z} and \ref{thm-quotient}, we reformulate the Hankel-Toeplitz matrix optimization problem~\eqref{optimization-problem-HT} as an optimization over equivalence classes on a quotient manifold:
\begin{equation}
\label{optimzation-problem-HT-quotient}
\begin{aligned}
\min_{\left[\m{Z}\right]\in\mathbb{C}^{p\times K}_*/\mathcal{O}^K}\;&\rho\left(\mathcal{P}_\Omega\left(\m{D}^{-2}\mathcal{H}^*\left(\beta_1\left(\left[\m{Z}\right]\right)\right)\right),\mathcal{P}_\Omega\left(\m{y}\right)\right),\\
\st&\left(\mathcal{I}-\mathcal{G}_1\right)\left(\beta_1\left(\left[\m{Z}\right]\right)\right)=\m{0},\;\left(\mathcal{I}-\mathcal{G}_2\right)\left(\beta_2\left(\left[\m{Z}\right]\right)\right)=\m{0},
\end{aligned}
\end{equation}
where $\beta_1\left(\left[\m{Z}\right]\right)=\m{Z}\m{Z}^\top$ and $\beta_2\left(\left[\m{Z}\right]\right)\m{Z}\m{Z}^H$, for any $\m{Z}\in\left[\m{Z}\right]$. In the absence of noise, we further reformulate~\eqref{optimzation-problem-HT-quotient} as the following unconstrained minimization problem on the quotient manifold $\mathbb{C}^{p\times K}_*/\mathcal{O}^K$:
\begin{equation}
\label{unconstrained-optimzation-problem-HT-quotient}
\begin{aligned}
\min_{\left[\m{Z}\right]\in\mathbb{C}^{p\times K}_*/\mathcal{O}^K}\;&\rho\left(\mathcal{P}_\Omega\left(\m{D}^{-2}\mathcal{H}^*\left(\beta_1\left(\left[\m{Z}\right]\right)\right)\right),\mathcal{P}_\Omega\left(\m{y}\right)\right)\\
&+\frac{\mu}{4}\left\|\left(\mathcal{I}-\mathcal{G}_1\right)\left(\beta_1\left(\left[\m{Z}\right]\right)\right)\right\|_{\text{F}}^2+\frac{\mu}{4}\left\|\left(\mathcal{I}-\mathcal{G}_2\right)\left(\beta_2\left(\left[\m{Z}\right]\right)\right)\right\|_{\text{F}}^2,
\end{aligned}
\end{equation}
where the regularization coefficient $\mu$ is chosen as the sampling ratio $\left|\Omega\right|/N$, as suggested in~\cite{wu2024fast}. It can be easily shown that both~\eqref{optimzation-problem-HT-quotient} and~\eqref{unconstrained-optimzation-problem-HT-quotient} have an optimal value of zero and share the same optimal solution.

Notably, if the original problem~\eqref{feasibility-problem} has a unique optimal solution, then the equivalence class optimization problem~\eqref{optimzation-problem-HT-quotient} or~\eqref{unconstrained-optimzation-problem-HT-quotient} also possesses a unique optimal solution, which captures the frequencies and cofficients precisely. 
We then aim to efficiently solve the unconstrained optimization problem~\eqref{unconstrained-optimzation-problem-HT-quotient} on the quotient manifold $\mathbb{C}^{p\times K}_*/\mathcal{O}^K$. Prior research suggests that under a well-structured Riemannian manifold framework, advanced optimization algorithms can improve iteration efficiency and potentially relax recovery conditions such as sample complexity or condition number requirements~\cite{vandereycken2013low,bian2024preconditioned,bian2024preconditioned-SAM,steinlechner2016riemannian}. To that end, a necessary task is to study the geometric properties of $\mathbb{C}^{p\times K}_*/\mathcal{O}^K$, which will be done in the ensuing subsection.

\subsection{Riemannian geometry of \texorpdfstring{$\mathbb{C}^{p\times K}_*/\mathcal{O}^K$}{}}

We now consider the Riemannian geometry of the quotient manifold $\mathbb{C}^{p\times K}_*/\mathcal{O}^K$. The set $\mathbb{C}^{p\times K}_*$ is referred to as the total space. We equip $\mathbb{C}^{p\times K}_*$ with a Riemannian metric:
\begin{equation}
\label{defn-metric}
\begin{aligned}
\widetilde{g}_{\m{Z}}\left({\m{\xi}}_{\uparrow_{\m{Z}}},{\m{\zeta}}_{\uparrow_{\m{Z}}}\right)=\Re\left\{\tr\left(\Re\left\{\m{Z}^H\m{Z}\right\}\widetilde{\m{\xi}}^H\widetilde{\m{\zeta}}\right)\right\}=\tr\left(\Re\left\{\m{Z}^H\m{Z}\right\}\Re\left\{\widetilde{\m{\xi}}^H\widetilde{\m{\zeta}}\right\}\right),
\end{aligned}
\end{equation}
where $\m{Z} \in \mathbb{C}^{p\times K}_*$ and $\widetilde{\m{\xi}},\widetilde{\m{\zeta}} \in \text{T}_{\m{Z}}\mathbb{C}^{p\times K}=\mathbb{C}^{p\times K}$. At each point $\m{Z}$, the tangent space $\text{T}_{\m{Z}}\mathbb{C}^{p\times K}$ can be decomposed into two complementary subspaces with respect to the Riemannian metric $\widetilde{g}_{\cdot}\left(\cdot,\cdot\right)$:
\begin{equation}
\begin{aligned}
\text{T}_{\m{Z}}\mathbb{C}^{p\times K}=\mathscr{H}_{\m{Z}}\oplus \mathscr{V}_{\m{Z}},
\end{aligned}
\end{equation}
where $\mathscr{V}_{\m{Z}}$ denotes the vertical space, consisting of directions that move $\m{Z}$ within $\left[\m{Z}\right]$, and $\mathscr{H}_{\m{Z}}$ denotes the horizontal space, defined as the orthogonal complement of $\mathscr{V}_{\m{Z}}$ under $\widetilde{g}_{\cdot}\left(\cdot,\cdot\right)$. We obtain from \cite[Proposition 9.3]{boumal2023introduction} that 
\begin{equation}
\begin{aligned}
\mathscr{V}_{\m{Z}}=\text{T}_{\m{Z}}\left[\m{Z}\right]=\text{T}_{\m{Z}}\m{Z}\mathcal{O}^K=\m{Z}\text{T}_{\m{I}_K}\mathcal{O}^K.
\end{aligned}
\end{equation}
The tangent space of $\mathcal{O}^K$ at $\m{I}_K$ is $\text{T}_{\m{I}_K}\mathcal{O}^K=\left\{\m{\Omega}\in\mathbb{R}^{K\times K}:\;\m{\Omega}^\top=-\m{\Omega}\right\}$, see~\cite{absil2008optimization}. Therefore, 
\begin{equation}
\label{vertical-space}
\begin{aligned}
\mathscr{V}_{\m{Z}}=\left\{\m{Z}\m{\Omega}\in\mathbb{C}^{p\times K}:\;\m{\Omega}^\top=-\m{\Omega},\m{\Omega}\in\mathbb{R}^{K\times K}\right\}.
\end{aligned}
\end{equation}
The horizontal space is then given by
\begin{equation}
\label{horizontal-space}
\begin{aligned}
\mathscr{H}_{\m{Z}}&=\left\{\widetilde{\m{\xi}}\in\mathbb{C}^{p\times K}:\;\widetilde{g}_{\m{Z}}\left(\widetilde{\m{\xi}},\widetilde{\m{\zeta}}\right)=0,\forall\widetilde{\m{\zeta}}\in V_{\m{Z}}\right\}\\
&=\left\{\widetilde{\m{\xi}}\in\mathbb{C}^{p\times K}:\;\tr\left(\Re\left\{\m{Z}^H\m{Z}\right\}\Re\left\{\widetilde{\m{\xi}}^H\m{Z}\right\}\m{\Omega}\right)=0, \text{  $\forall\m{\Omega}\in\mathbb{R}^{K\times K}$, $\m{\Omega}^\top=-\m{\Omega}$}\right\}\\
&=\left\{\widetilde{\m{\xi}}\in\mathbb{C}^{p\times K}:\;\Re\left\{\m{Z}^H\m{Z}\right\}\Re\left\{\widetilde{\m{\xi}}^H\m{Z}\right\}=\Re\left\{\widetilde{\m{\xi}}^H\m{Z}\right\}^\top\Re\left\{\m{Z}^H\m{Z}\right\}\right\},
\end{aligned}
\end{equation}
where the last equality follows from the orthogonality and complementarity between symmetric and skew-symmetric matrices. We provide an explicit formula for projecting onto the horizontal space $\mathscr{H}_{\m{Z}}$, stated in the following proposition.
\begin{prop}
\label{prop-projection-horizontal}
Let $\widetilde{\m{\xi}}\in\mathbb{C}^{p\times K}$. The projection of $\widetilde{\m{\xi}}$ onto the horizontal space $\mathscr{H}_{\m{Z}}$ is given by
\begin{equation}
\begin{aligned}
\mathcal{P}_{\mathscr{H}_{\m{Z}}}\left(\widetilde{\m{\xi}}\right)=\widetilde{\m{\xi}}-\m{Z}\m{\Omega},
\end{aligned}
\end{equation}
where 
\begin{equation}
\label{Mat-Omega}
\begin{aligned}
\displaystyle \m{\Omega}=\frac{1}{2}\left(\Re\left\{\m{Z}^H\m{Z}\right\}^{-1}\Re\left\{\m{Z}^H\widetilde{\m{\xi}}\right\}-\Re\left\{\widetilde{\m{\xi}}^H\m{Z}\right\}\Re\left\{\m{Z}^H\m{Z}\right\}^{-1}\right).
\end{aligned}
\end{equation}
\end{prop}

\begin{sloppypar}
\begin{proof}
Since $\mathscr{H}_{\m{Z}}$ and $\mathscr{V}_{\m{Z}}$ are orthogonal, it suffices to verify that $\m{Z}\m{\Omega}\in\mathscr{V}_{\m{Z}}$ and $\widetilde{\m{\xi}}-\m{Z}\m{\Omega}\in\mathscr{H}_{\m{Z}}$. First, observe that $\m{\Omega}$ in~\eqref{Mat-Omega} is skew-symmetric. From~\eqref{vertical-space}, this implies $\m{Z}\m{\Omega}\in\mathscr{V}_{\m{Z}}$. Next, we have that
\begin{equation}
\begin{aligned}
\Re\left\{\m{Z}^H\m{Z}\right\}\Re\left\{\left(\widetilde{\m{\xi}}-\m{Z}\m{\Omega}\right)^H\m{Z}\right\}=&\Re\left\{\m{Z}^H\m{Z}\right\}\Re\left\{\widetilde{\m{\xi}}^H\m{Z}+\m{\Omega}\m{Z}^H\m{Z}\right\}\\
=&\frac{1}{2}\left(\Re\left\{\m{Z}^H\m{Z}\right\}\Re\left\{\widetilde{\m{\xi}}^H\m{Z}\right\}+\Re\left\{\m{Z}^H\widetilde{\m{\xi}}\right\}\Re\left\{\m{Z}^H\m{Z}\right\}\right)
\end{aligned}
\end{equation}
is symmetric. By~\eqref{horizontal-space}, this shows $\widetilde{\m{\xi}}-\m{Z}\m{\Omega}\in\mathscr{H}_{\m{Z}}$. Therefore, $\widetilde{\m{\xi}}$ admits the orthogonal decomposition $\widetilde{\m{\xi}}=\left(\widetilde{\m{\xi}}-\m{Z}\m{\Omega}\right)+\m{Z}\m{\Omega}$, which completes the proof.
\end{proof}
\end{sloppypar}

Define the natural projection as $\Pi: \mathbb{C}_*^{p \times K} \rightarrow \mathbb{C}^{p\times K}_*/\mathcal{O}^K :\m{Z} \mapsto \left[\m{Z}\right]$. The map $\Pi\left(\cdot\right)$ is a smooth submersion, and its differential map defines a linear mapping from the tangent space of the total space to that of the quotient manifold. For every $\m{\xi}\in \text{T}_{\left[\m{Z}\right]}\mathbb{C}^{p\times K}_*/\mathcal{O}^K$, there exists a unique horizontal lift at $\m{Z}$, denoted by ${\m{\xi}}_{\uparrow_{\m{Z}}} \in \mathscr{H}_{\m{Z}}$, such that $\text{D}\Pi(\m{Z})\left[{\m{\xi}}_{\uparrow_{\m{Z}}}\right] = \m{\xi}$, see details in~\cite{absil2008optimization}. We establish in the following proposition the relationship between horizontal lifts of a tangent vector $\m{\xi}$ lifted at different representatives in $\left[\m{Z}\right]$.
\begin{prop}
\label{prop-horizontal lift}
For every $\left[\m{Z}\right]\in\mathbb{C}^{p\times K}_*/\mathcal{O}^K$, $\m{O}\in\mathcal{O}^K$, and tangent vector $\m{\xi}\in \text{T}_{\left[\m{Z}\right]}\mathbb{C}^{p\times K}_*/\mathcal{O}^K$, if ${\m{\xi}}_{\uparrow_{\m{Z}}}$ is the horizontal lift of $\m{\xi}$ at $\m{Z}$ and ${\m{\xi}}_{\uparrow_{\m{Z}\m{O}}}$ is the horizontal lift of $\m{\xi}$ at $\m{Z}\m{O}$, then
\begin{equation}
\begin{aligned}
{\m{\xi}}_{\uparrow_{\m{Z}\m{O}}}={\m{\xi}}_{\uparrow_{\m{Z}}}\m{O}.
\end{aligned}
\end{equation}
\end{prop}

\begin{sloppypar}
\begin{proof}
The proof follows the idea of \cite[Lemma~2.17]{zheng2022riemannian}, which considers a related Riemannian quotient manifold. Here, we adapt the argument to the specific quotient manifold in this work.

Since $\Pi\left(\m{Z}\right)=\Pi\left(\m{Z}\m{O}\right)$, differentiating with respect to $\m{Z}$ gives $\text{D}\Pi\left(\m{Z}\right)\left[\widetilde{\m{\xi}}\right]=\text{D}\Pi\left(\m{Z}\m{O}\right)\left[\widetilde{\m{\xi}}\m{O}\right],\;\forall\widetilde{\m{\xi}}\in\mathbb{C}^{p\times K}$. Applying this to ${\m{\xi}}_{\uparrow_{\m{Z}}}$ yields $\text{D}\Pi\left(\m{Z}\right)\left[{\m{\xi}}_{\uparrow_{\m{Z}}}\right]=\text{D}\Pi\left(\m{Z}\m{O}\right)\left[{\m{\xi}}_{\uparrow_{\m{Z}}}\m{O}\right]$. Thus, $\text{D}\Pi\left(\m{Z}\m{O}\right)\left[{\m{\xi}}_{\uparrow_{\m{Z}\m{O}}}\right]=\text{D}\Pi\left(\m{Z}\right)\left[{\m{\xi}}_{\uparrow_{\m{Z}}}\right]=\text{D}\Pi\left(\m{Z}\m{O}\right)\left[{\m{\xi}}_{\uparrow_{\m{Z}}}\m{O}\right]$, which indicates ${\m{\xi}}_{\uparrow_{\m{Z}\m{O}}}-{\m{\xi}}_{\uparrow_{\m{Z}}}\m{O}\in \mathscr{V}_{\m{Z}\m{O}}$. Moreover, since ${\m{\xi}}_{\uparrow_{\m{Z}}}\in \mathscr{H}_{\m{Z}}$, $\Re\left\{\m{Z}^H\m{Z}\right\}\Re\left\{{\m{\xi}}_{\uparrow_{\m{Z}}}^H\m{Z}\right\}$ is symmetric. It follows that $\Re\left\{\left(\m{Z}\m{O}\right)^H\m{Z}\m{O}\right\}\Re\left\{\left({\m{\xi}}_{\uparrow_{\m{Z}}}\m{O}\right)^H\m{Z}\m{O}\right\}=\m{O}^\top\Re\left\{\m{Z}^H\m{Z}\right\}\Re\left\{{\m{\xi}}_{\uparrow_{\m{Z}}}^H\m{Z}\right\}\m{O}$ is symmetric, hence ${\m{\xi}}_{\uparrow_{\m{Z}}}\m{O}\in \mathscr{H}_{\m{Z}\m{O}}$. Therefore, ${\m{\xi}}_{\uparrow_{\m{Z}\m{O}}}-{\m{\xi}}_{\uparrow_{\m{Z}}}\m{O}\in \mathscr{V}_{\m{Z}\m{O}}\cap \mathscr{H}_{\m{Z}\m{O}}=\left\{\m{0}\right\}$, which completes the proof.
\end{proof}
\end{sloppypar}

It follows from Proposition \ref{prop-horizontal lift} that for every $\left[\m{Z}\right]\in\mathbb{C}^{p\times K}_*/\mathcal{O}^K$, $\m{O}\in\mathcal{O}^K$, and tangent vectors $\m{\xi},\m{\zeta}\in \text{T}_{\left[\m{Z}\right]}\mathbb{C}^{p\times K}_*/\mathcal{O}^K$, if ${\m{\xi}}_{\uparrow_{\m{Z}}}$ and ${\m{\zeta}}_{\uparrow_{\m{Z}}}$ are the horizontal lifts of $\m{\xi}$ and $\m{\zeta}$ at $\m{Z}$, then their horizontal lifts at $\m{Z}\m{O}$ are given by ${\m{\xi}}_{\uparrow_{\m{Z}\m{O}}}={\m{\xi}}_{\uparrow_{\m{Z}}}\m{O}$ and ${\m{\zeta}}_{\uparrow_{\m{Z}\m{O}}}={\m{\zeta}}_{\uparrow_{\m{Z}}}\m{O}$, respectively. Hence,
\begin{equation}
\begin{aligned}
\widetilde{g}_{\m{Z}\m{O}}\left({\m{\xi}}_{\uparrow_{\m{Z}\m{O}}},{\m{\zeta}}_{\uparrow_{\m{Z}\m{O}}}\right)&=\tr\left(\m{O}^\top\Re\left\{\m{Z}^H\m{Z}\right\}\m{O}\m{O}^\top\Re\left\{{\m{\xi}}_{\uparrow_{\m{Z}}}^H{\m{\zeta}}_{\uparrow_{\m{Z}}}\right\}\m{O}\right)\\
&=\tr\left(\Re\left\{\m{Z}^H\m{Z}\right\}\Re\left\{{\m{\xi}}_{\uparrow_{\m{Z}}}^H{\m{\zeta}}_{\uparrow_{\m{Z}}}\right\}\right)\\
&=\widetilde{g}_{\m{Z}}\left({\m{\xi}}_{\uparrow_{\m{Z}}},{\m{\zeta}}_{\uparrow_{\m{Z}}}\right),
\end{aligned}
\end{equation}
which confirms that the metric is invariant under the choice of the representative $\m{Z}\in\left[\m{Z}\right]$ and thus induces a Riemannian metric on the quotient manifold.
For any $\left[\m{Z}\right]\in\mathbb{C}^{p\times K}_*/\mathcal{O}^K$ and any $\m{\xi},\m{\zeta}\in \text{T}_{\left[\m{Z}\right]}\mathbb{C}^{p\times K}_*/\mathcal{O}^K$, the Riemannian metric is defined as
\begin{equation}
g_{[\m{Z}]}(\m{\xi},\m{\zeta}) = \widetilde{g}_{\m{Z}}\left({\m{\xi}}_{\uparrow_{\m{Z}}},{\m{\zeta}}_{\uparrow_{\m{Z}}}\right).
\end{equation}

\begin{sloppypar}
To develop an iterative optimization algorithm on $\mathbb{C}^{p\times K}_*/\mathcal{O}^K$, we establish mechanisms for computing update directions and moving points along these directions, based on the concepts of retraction and vector transport. A retraction is a mapping that moves a point on the manifold in the direction of a given tangent vector, while a vector transport maps a tangent vector from one tangent space to another. We define the retraction $\mathcal{R}_{\cdot}\left(\cdot\right)$ on $\mathbb{C}^{p\times K}_*/\mathcal{O}^K$ via a specific retraction $\widetilde{\mathcal{R}}_{\cdot}\left(\cdot\right)$ on the total space $\mathbb{C}^{p\times K}_*$:
\begin{equation}
\label{induced-retraction}
\begin{aligned}
\mathcal{R}_{\left[\m{Z}\right]}\left(\m{\xi}\right)=\Pi\left(\widetilde{\mathcal{R}}_{\m{Z}}\left({\m{\xi}}_{\uparrow_{\m{Z}}}\right)\right),
\end{aligned}
\end{equation}
where $\left[\m{Z}\right]\in\mathbb{C}^{p\times K}_*/\mathcal{O}^K$, $\m{\xi}\in \text{T}_{\left[\m{Z}\right]}\mathbb{C}^{p\times K}_*/\mathcal{O}^K$, and $\m{Z}\in\left[\m{Z}\right]$ is a representative in the total space. As shown in~\cite{absil2008optimization},
\begin{equation}
\label{retraction-totalspace}
\begin{aligned}
\widetilde{\mathcal{R}}_{\m{Z}}\left({\m{\xi}}_{\uparrow_{\m{Z}}}\right)=\m{Z}+{\m{\xi}}_{\uparrow_{\m{Z}}}
\end{aligned}
\end{equation}
is a valid retraction on $\mathbb{C}^{p\times K}_*$, provided that $\m{Z}+{\m{\xi}}_{\uparrow_{\m{Z}}}$ remains full column rank. By Proposition \ref{prop-horizontal lift}, the expression in~\eqref{induced-retraction} is independent of the choice of the representative $\m{Z}\in\left[\m{Z}\right]$. Therefore, by \cite[Proposition~4.1.3]{absil2008optimization},
\begin{equation}
\label{defn-retraction}
\begin{aligned}
\mathcal{R}_{\left[\m{Z}\right]}\left(\m{\xi}\right)=\Pi\left(\widetilde{\mathcal{R}}_{\m{Z}}\left({\m{\xi}}_{\uparrow_{\m{Z}}}\right)\right)=\left[\m{Z}+{\m{\xi}}_{\uparrow_{\m{Z}}}\right],\forall \m{Z}\in\left[\m{Z}\right],
\end{aligned}
\end{equation}
defines a valid retraction on the quotient manifold. Next, we define the vector transport by differentiating the retraction~\cite{absil2008optimization}:
\begin{equation}
\label{defn-Vectortransport}
\begin{aligned}
\mathcal{V}\left(\left[\m{Z}\right],\m{\xi},\m{\zeta}\right)&=\text{D}\mathcal{R}_{\left[\m{Z}\right]}\left(\m{\xi}\right)\left[\m{\zeta}\right]\\
&=\text{D}\Pi\left(\widetilde{\mathcal{R}}_{\m{Z}}\left({\m{\xi}}_{\uparrow_{\m{Z}}}\right)\right)\left[dR_{\m{Z}}\left({\m{\xi}}_{\uparrow_{\m{Z}}}\right)\left[{\m{\zeta}}_{\uparrow_{\m{Z}}}\right]\right]\\
&=\text{D}\Pi\left(\m{Z}+{\m{\xi}}_{\uparrow_{\m{Z}}}\right)\left[\frac{d}{dt}R_{\m{Z}}\left({\m{\xi}}_{\uparrow_{\m{Z}}}+t{\m{\zeta}}_{\uparrow_{\m{Z}}}\right)\bigg|_{t=0}\right]\\
&=\text{D}\Pi\left(\m{Z}+{\m{\xi}}_{\uparrow_{\m{Z}}}\right)\left[\mathcal{P}_{\mathscr{H}_{\m{Z}+{\m{\xi}}_{\uparrow_{\m{Z}}}}}\left({\m{\zeta}}_{\uparrow_{\m{Z}}}\right)\right],
\end{aligned}
\end{equation}
where $\left[\m{Z}\right]\in\mathbb{C}^{p\times K}_*/\mathcal{O}^K$, and $\m{\xi},\m{\zeta}\in \text{T}_{\left[\m{Z}\right]}\mathbb{C}^{p\times K}_*/\mathcal{O}^K$.

\end{sloppypar}

\section{Riemanian Optimization Algorithm}
\label{Sec:HT-RCGD}

In this section, we propose an efficient algorithm on the quotient manifold to solve problem~\eqref{unconstrained-optimzation-problem-HT-quotient}. The primary technical difficulty arises from the non-compactness of both the quotient manifold and the sublevel sets of the objective function, which may lead to convergence issues. In addition, achieving fast convergence with low computational cost per iteration requires careful refinement of the algorithmic framework, including the design of iteration steps, parameter selection, and computational procedures. Although existing tools in the literature~\cite{vandereycken2013low,sato2022riemannian,absil2008optimization,wu2024fast} provide valuable references, it remains challenging to establish rigorous convergence guarantees while maintaining computational efficiency for the specific problems considered in this work.

\subsection{Objective function}
For algorithm design, it is essential to explicitly define the objective function. A feasible choice for $\rho\left(\cdot,\cdot\right)$ is the weighted least squares criterion~\cite{cai2018spectral}:
\begin{equation}
\label{defn-distance}
\begin{aligned}
\rho\left(\mathcal{P}_\Omega\left(\m{x}\right),\mathcal{P}_\Omega\left(\m{y}\right)\right)=\frac{1}{4}\left\|\mathcal{P}_\Omega\left(\m{D}\m{x}\right)-\mathcal{P}_\Omega\left(\m{D}\m{y}\right)\right\|_{2}^2,
\end{aligned}
\end{equation}
where the weight matrix $\m{D}$ is chosen based on the Hankel structure. Specifically, the $i$-th diagonal element of $\m{D}$ corresponds to the number of entries in the $i$-th skew-diagonal of a $p\times p$ matrix. The original objective function in~\eqref{unconstrained-optimzation-problem-HT-quotient} then becomes
\begin{equation}
\begin{aligned}
h\left(\left[\m{Z}\right]\right)=&\frac{1}{4}\left\|\mathcal{P}_\Omega\left(\m{D}^{-1}\mathcal{H}^*\left(\beta_1\left(\left[\m{Z}\right]\right)\right)\right)-\mathcal{P}_\Omega\left(\m{D}\m{y}\right)\right\|_{2}^2\\
&+\frac{\mu}{4}\left\|\left(\mathcal{I}-\mathcal{G}_1\right)\left(\beta_1\left(\left[\m{Z}\right]\right)\right)\right\|_{\text{F}}^2+\frac{\mu}{4}\left\|\left(\mathcal{I}-\mathcal{G}_2\right)\left(\beta_2\left(\left[\m{Z}\right]\right)\right)\right\|_{\text{F}}^2.
\end{aligned}
\end{equation}

To facilitate the convergence analysis of the proposed algorithm, we modify this objective function by adding a regularization term. This modification is motivated by the difficulty of establishing convergence guarantees without additional assumptions, such as Lipschitz continuity of the gradient or the existence of accumulation points, due to the quartic nature of the objective function induced by the quotient geometry. Inspired by~\cite{vandereycken2013low}, we incorporate a regularization term to ensure that the iterates remain within a compact subset, thereby guaranteeing the existence of accumulation points. We define the modified objective function as
\begin{equation}
\label{modified-objective-fun}
\widehat{h}\left(\left[\m{Z}\right]\right) = h\left(\left[\m{Z}\right]\right) + \lambda\psi\left(\left[\m{Z}\right]\right),
\end{equation}
where $\lambda > 0$ is a regularization parameter, and the regularization term $\psi(\cdot)$ is defined as
\begin{equation}
\label{regularization-term}
\psi\left(\left[\m{Z}\right]\right) = \frac{1}{2} \left( \left\|\m{Z}\right\|_{\text{F}}^2 + \left\|\m{Z}^\dag\right\|_{\text{F}}^2 \right), \quad \forall \m{Z} \in \left[\m{Z}\right].
\end{equation}
It is straightforward to verify that $\psi\left(\left[\m{Z}\right]\right)$ is invariant under the choice of representative $\m{Z} \in \left[\m{Z}\right]$, indicating that the regularization term is well-defined. According to Theorem \ref{prop-x2Z}, the optimal solution of problem~\eqref{unconstrained-optimzation-problem-HT-quotient} has full column rank, ensuring that $\psi(\cdot)$ remains bounded at the optimum. Therefore, we can select a sufficiently small regularization parameter $\lambda > 0$ (e.g., $\lambda = 10^{-8}$), such that the modified objective function yields an optimal solution arbitrarily close to that of the original. Consequently, the original objective function in~\eqref{unconstrained-optimzation-problem-HT-quotient} can be safely replaced by $\widehat{h}$ in~\eqref{modified-objective-fun}. The resulting optimization problem is
\begin{equation}
\label{modified-optimzation-problem}
\min_{\left[\m{Z}\right] \in \mathbb{C}^{p \times K}_*/\mathcal{O}^K} \; \widehat{h}\left(\left[\m{Z}\right]\right).
\end{equation}

\subsection{Riemannian conjugate gradient method}

Motivated by the Riemannian conjugate gradient (RCG) method, known for its rapid convergence in unconstrained optimization over Riemannian manifolds~\cite{sakai2021sufficient,sato2022riemannian,zheng2022riemannian}, we develop a tailored variant that exploits the specific structure of problem~\eqref{unconstrained-optimzation-problem-HT-quotient} for efficient implementation. The RCG method generalizes the classical conjugate gradient algorithm from Euclidean spaces to Riemannian manifolds by introducing three key components: replacing the Euclidean gradient with the Riemannian gradient, using vector transport to compute conjugate directions, and employing a retraction to update iterates along search directions.

Using the metric defined in~\eqref{defn-metric}, the Riemannian gradient of $\widehat{h}$ is characterized in the following proposition.
\begin{prop}
\label{prop-Riemannian-gradient}
The gradient of $\widehat{h}$ at $\left[\m{Z}\right]$ is characterized by its horizontal lift at $\m{Z}$ as
\begin{equation}
\label{horizontal-lift-gradient}
\begin{aligned}
{\text{grad}\,\widehat{h}\left(\left[\m{Z}\right]\right)}_{\uparrow_{\m{Z}}} &= \Big( \mathcal{H}\m{D}^{-1}\mathcal{P}_\Omega^*\left(\mathcal{P}_\Omega\left(\m{D}^{-1}\mathcal{H}^*\left(\m{Z}\m{Z}^\top\right)\right) - \mathcal{P}_\Omega\left(\m{D}\m{y}\right)\right)\overline{\m{Z}} \\
&\quad + \mu\left(\mathcal{I} - \mathcal{G}_1\right)\left(\m{Z}\m{Z}^\top\right)\overline{\m{Z}} + \mu\left(\mathcal{I} - \mathcal{G}_2\right)\left(\m{Z}\m{Z}^H\right)\m{Z} \\
&\quad + \lambda\m{U}\left(\m{\Sigma} - \m{\Sigma}^{-2}\right)\m{V}^H \Big)\Re\left\{\m{Z}^H\m{Z}\right\}^{-1},
\end{aligned}
\end{equation}
where $\m{Z} = \m{U} \m{\Sigma} \m{V}^H$ is the SVD.
\end{prop}

\begin{sloppypar}
\begin{proof}
For any $\widetilde{\m{\xi}} \in \mathbb{C}^{p \times K}$, we have
\begin{equation}
\begin{aligned}
\widetilde{g}_{\m{Z}}\left({\text{grad}\,\widehat{h}\left(\left[\m{Z}\right]\right)}_{\uparrow_{\m{Z}}}, \widetilde{\m{\xi}}\right) = \left\langle \nabla \left(\widehat{h}\circ\Pi\right)\left(\m{Z}\right), \widetilde{\m{\xi}} \right\rangle,
\end{aligned}
\end{equation}
where $\widehat{h}\circ\Pi$ denotes the composition of $\widehat{h}$ and $\Pi$. This implies 
\begin{equation}
\begin{aligned}
\Re\left\{\tr\left(\Re\left\{\m{Z}^H\m{Z}\right\}{\text{grad}\,\widehat{h}\left(\left[\m{Z}\right]\right)}_{\uparrow_{\m{Z}}}^H\widetilde{\m{\xi}}\right)\right\} = \left\langle \nabla \left(\widehat{h}\circ\Pi\right)\left(\m{Z}\right), \widetilde{\m{\xi}} \right\rangle.
\end{aligned}
\end{equation}
Hence, ${\text{grad}\,\widehat{h}\left(\left[\m{Z}\right]\right)}_{\uparrow_{\m{Z}}}=\nabla \left(\widehat{h}\circ\Pi\right)\left(\m{Z}\right)\Re\left\{\m{Z}^H\m{Z}\right\}^{-1}$. Moreover,
\begin{equation}
\begin{aligned}
\left(\widehat{h}\circ\Pi\right)\left(\m{Z}\right)=\left({h}\circ\Pi\right)\left(\m{Z}\right)+\lambda\left(\psi\circ\Pi\right)\left(\m{Z}\right).
\end{aligned}
\end{equation}
The Euclidean gradient of $\left({h}\circ\Pi\right)$ is given by~\cite{wu2024fast}: $\nabla \left({h}\circ\Pi\right)\left(\m{Z}\right) = \mathcal{H}\m{D}^{-1}\mathcal{P}_\Omega^*\left(\mathcal{P}_\Omega\left(\m{D}^{-1}\mathcal{H}^*\left(\m{Z}\m{Z}^\top\right)\right) - \mathcal{P}_\Omega\left(\m{D}\m{y}\right)\right)\overline{\m{Z}}+ \mu\left(\mathcal{I} - \mathcal{G}_1\right)\left(\m{Z}\m{Z}^\top\right)\overline{\m{Z}} + \mu\left(\mathcal{I} - \mathcal{G}_2\right)\left(\m{Z}\m{Z}^H\right)\m{Z}$. Furthermore, by \cite[Theorem 4.3]{golub1973differentiation}, the gradient of $\left(\psi\circ\Pi\right)$ is $\nabla\left(\psi\circ\Pi\right)\left(\m{Z}\right) = \m{U}\left(\m{\Sigma} - \m{\Sigma}^{-2}\right)\m{V}^H$. Combining these results completes the proof.
\end{proof}
\end{sloppypar}

To incorporate historical gradient information and refine the search direction beyond the steepest descent, we compute a conjugate direction at each iteration. Specifically, at iteration $t$, the conjugate direction is given by
\begin{equation}
\begin{aligned}
\widetilde{\m{\eta}}_t=-{\text{grad}\widehat{h}\left(\left[\m{Z}_{t-1}\right]\right)}_{\uparrow_{\m{Z}_{t-1}}}+\beta_t\mathcal{P}_{\mathscr{H}_{\m{Z}_{t-1}}}\left(\widetilde{\m{\xi}}_{t-1}\right),
\end{aligned}
\end{equation}
where $\m{Z}_{t-1}$ denotes the iterate at step $t-1$, $\beta_t$ is a scalar weighting the previous search direction, and $\widetilde{\m{\xi}}_{t-1}$ denotes the update direction at step $t-1$. We compute $\beta_t$ using the Polak-Ribi{\'e}re formula~\cite{vandereycken2013low}:
\begin{equation}
\label{update-beta}
\begin{aligned}
&\beta_t\\
=&\frac{\widetilde{g}_{\m{Z}_{t-1}}\left({\text{grad}\widehat{h}\left(\left[\m{Z}_{t-1}\right]\right)}_{\uparrow_{\m{Z}_{t-1}}},{\text{grad}\widehat{h}\left(\left[\m{Z}_{t-1}\right]\right)}_{\uparrow_{\m{Z}_{t-1}}}-\mathcal{P}_{\mathscr{H}_{\m{Z}_{t-1}}}\left({\text{grad}\widehat{h}\left(\left[\m{Z}_{t-2}\right]\right)}_{\uparrow_{\m{Z}_{t-2}}}\right)\right)}{\widetilde{g}_{\m{Z}_{t-2}}\left({\text{grad}\widehat{h}\left(\left[\m{Z}_{t-2}\right]\right)}_{\uparrow_{\m{Z}_{t-2}}},{\text{grad}\widehat{h}\left(\left[\m{Z}_{t-2}\right]\right)}_{\uparrow_{\m{Z}_{t-2}}}\right)}.
\end{aligned}
\end{equation}

To ensure global convergence, we impose a safeguard on the update direction~\cite{absil2008optimization}. Specifically, we define the update direction $\widetilde{\m{\xi}}_t$ as
\begin{equation}
\label{update direction}
\widetilde{\m{\xi}}_t=
\begin{cases}
\widetilde{\m{\eta}}_t,\;&\text{if $\widetilde{g}_{\m{Z}_t}\left(\widetilde{\m{\eta}}_t,-{\text{grad}\widehat{h}\left(\left[\m{Z}_{t-1}\right]\right)}_{\uparrow_{\m{Z}_{t-1}}}\right)>c$ };\\
-{\text{grad}\widehat{h}\left(\left[\m{Z}_{t-1}\right]\right)}_{\uparrow_{\m{Z}_{t-1}}},\;&\text{otherwise},\\
\end{cases}
\end{equation}
where $c$ is a positive scalar controlling the correlation between the update direction and the negative gradient. The criterion in~\eqref{update direction} enforces a gradient-related condition, ensuring that any accumulation point of the iterates is a critical point, provided an appropriate step size is used~\cite{absil2008optimization}.

We now specify the step size selection. The next iterate is updated via the retraction on the quotient manifold $\mathbb{C}^{p\times K}_*/\mathcal{O}^K$ defined in~\eqref{defn-retraction}:
\begin{equation}
\label{update}
\begin{aligned}
\m{Z}_{t}&=\widetilde{\mathcal{R}}_{\m{Z}_{t-1}}\left(\alpha_t\widetilde{\m{\xi}}_t\right).
\end{aligned}
\end{equation}
The step size $\alpha_t$ is selected to satisfy the Armijo condition~\cite{ring2012optimization}:
\begin{equation}
\label{Armijo-condition}
\begin{aligned}
\widehat{h}\left(\left[\m{Z}_{t-1}\right]\right)-\widehat{h}\left(\left[\widetilde{\mathcal{R}}_{\m{Z}_{t-1}}\left(\alpha_t\widetilde{\m{\xi}}_t\right)\right]\right)&\geq-C\alpha_t\widetilde{g}_{\m{Z}_{t-1}}\left({\text{grad}\widehat{h}\left(\left[\m{Z}_{t-1}\right]\right)}_{\uparrow_{\m{Z}_{t-1}}},\widetilde{\m{\xi}}_t\right),
\end{aligned}
\end{equation}
where $C\in\left(0,1\right)$ is a constant. As established in \cite[Proposition 2.1]{sato2015new}, such a step size exists because the objective function $\widehat{h}$ is smooth and bounded below. We initialize the step size search with $\alpha_t^{\left(0\right)}$, computed as
\begin{equation}
\label{opt-alpha}
\begin{aligned}
\alpha_t^{\left(0\right)}=\arg\min_{\alpha>0}\;{h}\left(\left[\widetilde{\mathcal{R}}_{\m{Z}_{t-1}}\left(\alpha\widetilde{\m{\xi}}_t\right)\right]\right).
\end{aligned}
\end{equation}
Note that
\begin{equation}
\begin{aligned}
&{h}\left(\left[\widetilde{\mathcal{R}}_{\m{Z}}\left(\alpha\widetilde{\m{\eta}}\right)\right]\right)\\
=&\frac{1}{4}\left\|\mathcal{P}_\Omega\left(\m{D}^{-1}\mathcal{H}^*\left(\left(\m{Z}+\alpha\widetilde{\m{\eta}}\right)\left(\m{Z}+\alpha\widetilde{\m{\eta}}\right)^\top\right)\right)-\mathcal{P}_\Omega\left(\m{D}\m{y}\right)\right\|_{2}^2\\
&+\frac{\mu}{4}\left\|\left(\mathcal{I}-\mathcal{G}_1\right)\left(\left(\m{Z}+\alpha\widetilde{\m{\eta}}\right)\left(\m{Z}+\alpha\widetilde{\m{\eta}}\right)^\top\right)\right\|_{\text{F}}^2+\frac{\mu}{4}\left\|\left(\mathcal{I}-\mathcal{G}_2\right)\left(\left(\m{Z}+\alpha\widetilde{\m{\eta}}\right)\left(\m{Z}+\alpha\widetilde{\m{\eta}}\right)^H\right)\right\|_{\text{F}}^2
\end{aligned}
\end{equation}
is a quartic polynomial in $\alpha$. Denote $\phi_t\left(\alpha\right)={h}\left(\left[\widetilde{\mathcal{R}}_{\m{Z}_{t-1}}\left(\alpha\widetilde{\m{\xi}}\right)\right]\right)$ as the objective function in~\eqref{opt-alpha}. The initial guess $\alpha_t^{\left(0\right)}$ is determined by finding the positive root of its derivative $\phi_t'\left(\alpha\right)$. We then perform backtracking over the interval $\left(0,\alpha_t^{\left(0\right)}\right)$ to identify the final step size $\alpha_t$ satisfying the Armijo condition in~\eqref{Armijo-condition}.

The complete algorithm is termed as Hankel-Toeplitz Riemannian Conjugate Gradient Descent (HT-RCGD) algorithm and is summarized in Algorithm 4.1. HT-RCGD can be viewed as a refinement of the existing HT-PGD algorithm~\cite{wu2024fast}, which is developed under the framework of Euclidean geometry. In contrast, HT-RCGD adopts a Riemannian conjugate gradient method that exploits the underlying manifold geometry to enhance iteration efficiency.

\begin{algorithm}[htbp]
\caption{HT-RCGD Algorithm}
\label{alg1}
\textbf{Input:} Observation index set $\Omega$, observed signal $\mathcal{P}_\Omega(\m{y})$, spectral sparsity level $K$, parameters $c$ and $C$.\\
\textbf{Output:} Estimated signal $\m{y}$.\\
\hspace*{0.02in}1: Initialize $\left\{\m{Z}_0, \widetilde{\m{\xi}}_0\right\}$;\\
\hspace*{0.02in}2: \textbf{for} $t = 1, \cdots$ \textbf{do}\\
\hspace*{0.02in}3: \hspace*{0.26in} Compute Riemannian gradient: \hfill$\triangleright$ See Algorithm 4.2\\
\hspace*{0.92in} $\displaystyle {\text{grad}\,\widehat{h}\left(\left[\m{Z}_{t-1}\right]\right)}_{\uparrow_{\m{Z}_{t-1}}}$;\\
\hspace*{0.02in}4: \hspace*{0.26in} Compute conjugate direction: \\
\hspace*{0.92in} evaluate $\beta_t$ in~\eqref{update-beta},\\
\hspace*{0.92in} $\displaystyle \widetilde{\m{\eta}}_t=-{\text{grad}\widehat{h}\left(\left[\m{Z}_{t-1}\right]\right)}_{\uparrow_{\m{Z}_{t-1}}}+\beta_t\mathcal{P}_{\mathscr{H}_{\m{Z}_{t-1}}}\left(\widetilde{\m{\xi}}_{t-1}\right)$;\\
\hspace*{0.02in}5: \hspace*{0.26in} Compute update direction:\\
\hspace*{0.92in} \textbf{if} $\widetilde{g}_{\m{Z}_t}\left(\widetilde{\m{\eta}}_t,-{\text{grad}\widehat{h}\left(\left[\m{Z}_{t-1}\right]\right)}_{\uparrow_{\m{Z}_{t-1}}}\right)>c$ \textbf{then}\\
\hspace*{1.38in} $\widetilde{\m{\xi}}_t = \widetilde{\m{\eta}}_t$;\\
\hspace*{0.92in} \textbf{else}\\
\hspace*{1.38in} $\widetilde{\m{\xi}}_t = -{\text{grad}\widehat{h}\left(\left[\m{Z}_{t-1}\right]\right)}_{\uparrow_{\m{Z}_{t-1}}}$;\\
\hspace*{0.92in} \textbf{end if}\\
\hspace*{0.02in}6: \hspace*{0.26in} Compute initial step size: \hfill$\triangleright$ See Algorithm 4.3\\
\hspace*{0.92in} $\displaystyle \alpha_t^{(0)} = \arg\min_{\alpha>0}\;{h}\left(\left[\widetilde{\mathcal{R}}_{\m{Z}_{t-1}}\left(\alpha\widetilde{\m{\xi}}_t\right)\right]\right)$;\\
\hspace*{0.02in}7: \hspace*{0.26in} Perform backtracking search to find $\alpha_t$ satisfying the Armijo condition~\eqref{Armijo-condition}; \\
\hspace*{0.02in}\hfill$\triangleright$ See Algorithm 4.4\\
\hspace*{0.02in}8: \hspace*{0.26in} Update $\m{Z}_t = \widetilde{\mathcal{R}}_{\m{Z}_{t-1}}\left(\alpha_t\widetilde{\m{\xi}}_t\right)$;\\
\hspace*{0.02in}9: \textbf{end for}\\
\hspace*{0.02in}10: Compute the final estimate $\m{y} = \left(\mathcal{H}^* \mathcal{H}\right)^{-1} \mathcal{H}^*\left(\m{Z} \m{Z}^\top\right)$.
\end{algorithm}

\subsection{Fast implementation and computational complexity}

We present an efficient implementation of the HT-RCGD algorithm by leveraging the problem structure, and analyze its computational complexity with respect to $N$ and $K$. Each iteration primarily involves two steps: (1) computing the update direction and (2) computing the step size.

\subsubsection{Computing the update direction}
\begin{sloppypar}
Let $\m{Z}=\begin{bmatrix}\m{z}_1&\cdots&\m{z}_K\end{bmatrix}$ denote the current iterate. The conjugate direction $\widetilde{\m{\eta}}$ is a linear combination of the Riemannian gradient ${\text{grad}\,\widehat{h}\left(\left[\m{Z}\right]\right)}_{\uparrow_{\m{Z}}}$ and the transported previous update direction $\widetilde{\m{\xi}}_{-1}\in \mathscr{H}_{\m{Z}_{-1}}$ to the current horizontal space $\mathscr{H}_{\m{Z}}$ via $\mathcal{P}_{\mathscr{H}_{\m{Z}}}\left(\widetilde{\m{\xi}}_{-1}\right)$ scaled by a weighting parameter $\beta$. According to Proposition \ref{prop-Riemannian-gradient}, the Riemannian gradient can be expressed as
\begin{equation}
\begin{aligned}
{\text{grad}\,\widehat{h}\left(\left[\m{Z}\right]\right)}_{\uparrow_{\m{Z}}}=&\Bigg(\mathcal{H}\m{D}^{-1}\mathcal{P}_\Omega^*\left(\mathcal{P}_\Omega\left(\m{D}^{-1}\sum_{k=1}^K\mathcal{H}^*\left(\m{z}_k\m{z}_k^\top\right)\right)-\mathcal{P}_\Omega\left(\m{D}\m{y}\right)\right)\overline{\m{Z}}\\
&+\mu\m{Z}\m{Z}^\top\overline{\m{Z}}+\mu\m{Z}\m{Z}^H\m{Z}\\
&-\mu\sum_{k=1}^K\mathcal{G}_1\left(\m{z}_k\m{z}_k^\top\right)\overline{\m{Z}}-\mu\sum_{k=1}^K\mathcal{G}_2\left(\m{z}\m{z}_k^H\right)\m{Z}\\
&+\lambda\m{U}\left(\m{\Sigma} - \m{\Sigma}^{-2}\right)\m{V}^H\Bigg)\Re\left\{\m{Z}^H\m{Z}\right\}^{-1},
\end{aligned}
\end{equation}
where $\m{Z} = \m{U} \m{\Sigma} \m{V}^H$ is the SVD. The operations $\mathcal{H}^*\left(\m{z}_k\m{z}_k^\top\right)$, $\mathcal{G}_1\left(\m{z}_k\m{z}_k^\top\right)$, $\mathcal{G}_2\left(\m{z}\m{z}_k^H\right)$ can be efficiently implemented using FFT-based fast convolution, each requiring $O\left(N\log N\right)$ operations per term~\cite{lu2015fast}. Summing across $K$ terms results in a total cost of $O\left(KN\log N\right)$ floating-point operations (flops). Matrix-vector multiplications involving the circulant matrices $\sum_{k=1}^K\mathcal{H}^*\left(\m{z}_k\m{z}_k^\top\right)$, $\sum_{k=1}^K\mathcal{G}_1\left(\m{z}_k\m{z}_k^\top\right)$, and $\sum_{k=1}^K\mathcal{G}_2\left(\m{z}\m{z}_k^H\right)$ are also accelerated using the FFT~\cite{lu2015fast}. Furthermore, by rearranging operations, we compute $\m{Z}\m{Z}^\top\overline{\m{Z}}+\mu\m{Z}\m{Z}^H\m{Z}$ using $O\left(K^2N\right)$ flops. The full Riemannian gradient evaluation is summarized in Algorithm 4.2, with a total cost of $O\left(KN\log N+K^2N\right)$ flops. 
\end{sloppypar}
\begin{algorithm}[htbp]
\label{alg-Riemannian gradient}
    \caption{Calculate the Riemannian gradient} 
    {\bf Input:}
    $\displaystyle \m{Z}\in\mathbb{C}^{p\times K}_*$.\\
    {\bf Output:} 
    Riemannian gradient $\displaystyle {\text{grad}\,\widehat{h}\left(\left[\m{Z}\right]\right)}_{\uparrow_{\m{Z}}}$.\\
    \hspace*{0.02in} 1: $\displaystyle \m{g}_1=\mathcal{H}\m{D}^{-1}\mathcal{P}_\Omega^*\left(\mathcal{P}_\Omega\left(\m{D}^{-1}\sum_{k=1}^K\mathcal{H}^*\left(\m{z}_k\m{z}_k^\top\right)\right)-\mathcal{P}_\Omega\left(\m{D}\m{y}\right)\right)\overline{\m{Z}}$; \\
\hspace*{0.02in}\hfill$\triangleright$ $\#$ $O\left(KN\log N+K^2N\right)$flops\\
    \hspace*{0.02in} 2: $\displaystyle \m{g}_2=\mu\m{Z}\m{Z}^\top\overline{\m{Z}}+\mu\m{Z}\m{Z}^H\m{Z}$; \hfill$\triangleright$ $\#$ $O\left(K^2N\right)$flops\\
    \hspace*{0.02in} 3: $\displaystyle \m{g}_3=\mu\sum_{k=1}^K\mathcal{G}_1\left(\m{z}_k\m{z}_k^\top\right)\overline{\m{Z}}+\mu\sum_{k=1}^K\mathcal{G}_2\left(\m{z}\m{z}_k^H\right)\m{Z}$; \hfill$\triangleright$ $\#$ $O\left(KN\log N\right)$flops\\
    \hspace*{0.02in} 4: Compute the SVD: $\displaystyle \m{Z}=\m{U}\m{\Sigma}\m{V}^H$; \hfill$\triangleright$ $\#$ $O\left(K^2N\right)$flops\\
    \hspace*{0.02in} 5: $\displaystyle \m{g}_4=\lambda\m{U}\left(\m{\Sigma} - \m{\Sigma}^{-2}\right)\m{V}^H$; \hfill$\triangleright$ $\#$ $O\left(K^2N\right)$flops\\
    \hspace*{0.02in} 6: $\displaystyle {\text{grad}\,\widehat{h}\left(\left[\m{Z}\right]\right)}_{\uparrow_{\m{Z}}}=\left(\m{g}_1+\m{g}_2+\m{g}_3+\m{g}_4\right)\Re\left\{\m{Z}^H\m{Z}\right\}^{-1}$. \hfill$\triangleright$ $\#$ $O\left(K^2N\right)$flops
\end{algorithm}

The vector transport is computed via small-scale matrix inversions and multiplications, with a computational cost of $O(K^2N)$ flops. The weighting parameter $\beta$ is computed using the Riemannian metric and vector transport. The metric $\widetilde{g}_{\m{Z}}\left(\widetilde{\m{\xi}},\widetilde{\m{\zeta}}\right) = \tr\left(\Re\left\{\m{Z}^H\m{Z}\right\}\Re\left\{\widetilde{\m{\xi}}^H\widetilde{\m{\zeta}}\right\}\right)$ is computed with $O\left(K^2N\right)$ flops. Hence, the overall cost of computing $\beta$ is $O\left(K^2N\right)$ flops.

\begin{sloppypar}
Combining all components, the conjugate direction is computed as $\widetilde{\m{\eta}} = -{\text{grad}\,\widehat{h}\left(\left[\m{Z}\right]\right)}_{\uparrow_{\m{Z}}} + \beta \mathcal{P}_{\mathscr{H}_{\m{Z}}}\left(\widetilde{\m{\xi}}_{-1}\right)$, which is then used to determine the new update direction $\widetilde{\m{\xi}}$ as in~\eqref{update direction}.
\end{sloppypar}

\subsubsection{Computing the step size}

\begin{sloppypar}
Given the current iterate $\m{Z}$ and the update direction $\widetilde{\m{\xi}}$, the objective function $h\left(\left[\widetilde{\mathcal{R}}_{\m{Z}}\left(\alpha\widetilde{\m{\xi}}\right)\right]\right)$ in~\eqref{opt-alpha} can be expressed as a quartic polynomial:
\begin{equation}
\begin{aligned}
\phi\left(\alpha\right)=c_0+c_1\alpha+c_2\alpha^2+c_3\alpha^3+c_4\alpha^4,
\end{aligned}
\end{equation}
where
\begin{equation}
\label{coff-stepsize}
\begin{aligned}
c_0=&\frac{1}{4}\left\|\mathcal{P}_\Omega\left(\m{D}^{-1}\mathcal{H}^*\left(\m{Z}\m{Z}^\top\right)\right)-\mathcal{P}_\Omega\left(\m{D}\m{y}\right)\right\|_{2}^2\\
&+\frac{\mu}{4}\left\|\left(\mathcal{I}-\mathcal{G}_1\right)\left(\m{Z}\m{Z}^\top\right)\right\|_{\text{F}}^2+\frac{\mu}{4}\left\|\left(\mathcal{I}-\mathcal{G}_2\right)\left(\m{Z}\m{Z}^H\right)\right\|_{\text{F}}^2,\\
c_1=&\frac{1}{2}\left\langle\mathcal{P}_\Omega\left(\m{D}^{-1}\mathcal{H}^*\left(\m{Z}\widetilde{\m{\xi}}^\top+\widetilde{\m{\xi}}\m{Z}^\top\right)\right),\mathcal{P}_\Omega\left(\m{D}^{-1}\mathcal{H}^*\left(\m{Z}\m{Z}^\top\right)\right)-\mathcal{P}_\Omega\left(\m{D}\m{y}\right)\right\rangle\\
&+\frac{\mu}{2}\left\langle\left(\mathcal{I}-\mathcal{G}_1\right)\left(\m{Z}\widetilde{\m{\xi}}^\top+\widetilde{\m{\xi}}\m{Z}^\top\right),\m{Z}\m{Z}^\top\right\rangle+\frac{\mu}{2}\left\langle\left(\mathcal{I}-\mathcal{G}_2\right)\left(\m{Z}\widetilde{\m{\xi}}^H+\widetilde{\m{\xi}}\m{Z}^H\right),\m{Z}\m{Z}^H\right\rangle,\\
c_2=&\frac{1}{2}\left\langle\mathcal{P}_\Omega\left(\m{D}^{-1}\mathcal{H}^*\left(\m{Z}\widetilde{\m{\xi}}^\top\right)\right),\mathcal{P}_\Omega\left(\m{D}^{-1}\mathcal{H}^*\left(\widetilde{\m{\xi}}\m{Z}^\top\right)\right)\right\rangle\\
&+\frac{1}{2}\left\langle\mathcal{P}_\Omega\left(\m{D}^{-1}\mathcal{H}^*\left(\widetilde{\m{\xi}}\widetilde{\m{\xi}}^\top\right)\right),\mathcal{P}_\Omega\left(\m{D}^{-1}\mathcal{H}^*\left(\m{Z}\m{Z}^\top\right)\right)-\mathcal{P}_\Omega\left(\m{D}\m{y}\right)\right\rangle\\
&+\frac{1}{4}\left(\left\|\mathcal{P}_\Omega\left(\m{D}^{-1}\mathcal{H}^*\left(\m{Z}\widetilde{\m{\xi}}^\top\right)\right)\right\|_{\text{F}}^2+\left\|\mathcal{P}_\Omega\left(\m{D}^{-1}\mathcal{H}^*\left(\widetilde{\m{\xi}}\m{Z}^\top\right)\right)\right\|_{\text{F}}^2\right)\\
&+\frac{\mu}{2}\left\langle\left(\mathcal{I}-\mathcal{G}_1\right)\left(\m{Z}\widetilde{\m{\xi}}^\top\right),\widetilde{\m{\xi}}\m{Z}^\top\right\rangle+\frac{\mu}{2}\left\langle\left(\mathcal{I}-\mathcal{G}_1\right)\left(\widetilde{\m{\xi}}\widetilde{\m{\xi}}^\top\right),\m{Z}\m{Z}^\top\right\rangle\\
&+\frac{\mu}{4}\left(\left\|\left(\mathcal{I}-\mathcal{G}_1\right)\left(\m{Z}\widetilde{\m{\xi}}^\top\right)\right\|_{\text{F}}^2+\left\|\left(\mathcal{I}-\mathcal{G}_1\right)\left(\widetilde{\m{\xi}}\m{Z}^\top\right)\right\|_{\text{F}}^2\right)\\
&+\frac{\mu}{2}\left\langle\left(\mathcal{I}-\mathcal{G}_2\right)\left(\m{Z}\widetilde{\m{\xi}}^H\right),\widetilde{\m{\xi}}\m{Z}^H\right\rangle+\frac{\mu}{2}\left\langle\left(\mathcal{I}-\mathcal{G}_2\right)\left(\widetilde{\m{\xi}}\widetilde{\m{\xi}}^H\right),\m{Z}\m{Z}^H\right\rangle\\
&+\frac{\mu}{4}\left(\left\|\left(\mathcal{I}-\mathcal{G}_2\right)\left(\m{Z}\widetilde{\m{\xi}}^H\right)\right\|_{\text{F}}^2+\left\|\left(\mathcal{I}-\mathcal{G}_2\right)\left(\widetilde{\m{\xi}}\m{Z}^H\right)\right\|_{\text{F}}^2\right),\\
c_3=&\frac{1}{2}\left\langle\mathcal{P}_\Omega\left(\m{D}^{-1}\mathcal{H}^*\left(\m{Z}\widetilde{\m{\xi}}^\top+\widetilde{\m{\xi}}\m{Z}^\top\right)\right),\mathcal{P}_\Omega\left(\m{D}^{-1}\mathcal{H}^*\left(\widetilde{\m{\xi}}\widetilde{\m{\xi}}^\top\right)\right)\right\rangle\\
&+\frac{1}{2}\left\langle\left(\mathcal{I}-\mathcal{G}_1\right)\left(\m{Z}\widetilde{\m{\xi}}^\top+\widetilde{\m{\xi}}\m{Z}^\top\right),\widetilde{\m{\xi}}\widetilde{\m{\xi}}^\top\right\rangle+\frac{1}{2}\left\langle\left(\mathcal{I}-\mathcal{G}_2\right)\left(\m{Z}\widetilde{\m{\xi}}^H+\widetilde{\m{\xi}}\m{Z}^H\right),\widetilde{\m{\xi}}\widetilde{\m{\xi}}^H\right\rangle,\\
c_4=&\frac{1}{4}\left\|\mathcal{P}_\Omega\left(\m{D}^{-1}\mathcal{H}^*\left(\widetilde{\m{\xi}}\widetilde{\m{\xi}}^\top\right)\right)\right\|_{2}^2+\frac{\mu}{4}\left\|\left(\mathcal{I}-\mathcal{G}_1\right)\left(\widetilde{\m{\xi}}\widetilde{\m{\xi}}^\top\right)\right\|_{\text{F}}^2+\frac{\mu}{4}\left\|\left(\mathcal{I}-\mathcal{G}_2\right)\left(\widetilde{\m{\xi}}\widetilde{\m{\xi}}^H\right)\right\|_{\text{F}}^2.
\end{aligned}
\end{equation}
The initial step size $\alpha^{\left(0\right)}$ is chosen as the smallest positive real root of its derivative:
\begin{equation}
\begin{aligned}
\phi'\left(\alpha\right)=c_1+2c_2\alpha+3c_3\alpha^2+4c_4\alpha^3.
\end{aligned}
\end{equation}
Using a similar approach to that employed for computing the Riemannian gradient, these coefficients can be efficiently evaluated using FFTs. The initial estimate of the step size is computed in Algorithm 4.3, with a total computational complexity of $O\left(K^2N+KN\log N\right)$ flops.
\begin{algorithm}[htbp]
\label{alg-initial guess}
    \caption{Calculate the initial step size} 
    {\bf Input:}
    $\displaystyle \m{Z}\in\mathbb{C}^{p\times K}_*$, and horizontal vector $\displaystyle \widetilde{\m{\xi}}\in \mathscr{H}_{\m{Z}}$.\\
    {\bf Output:} 
    Initial step size $\displaystyle \alpha^{(0)} = \arg\min_{\alpha > 0} h\left(\left[\widetilde{\mathcal{R}}_{\m{Z}}\left(\alpha \widetilde{\m{\xi}}\right)\right]\right)$.\\
    \hspace*{0.02in} 1: Compute coefficients $\displaystyle \left\{c_1,c_2,c_3,c_4\right\}$ from~\eqref{coff-stepsize}; \hfill$\triangleright$ $\#$ $O\left(KN\log N+K^2N\right)$flops\\
    \hspace*{0.02in} 2: Solve for $\alpha^{\left(0\right)}$ as the smallest real positive root of $\phi'\left(\alpha\right)$.
\end{algorithm}
\end{sloppypar}

Next, we refine the step size by enforcing the Armijo condition. The modified objective function is defined as:
\begin{equation}
\begin{aligned}
\widehat{h}\left(\left[\widetilde{\mathcal{R}}_{\m{Z}}\left(\alpha \widetilde{\m{\xi}}\right)\right]\right)=&{h}\left(\left[\widetilde{\mathcal{R}}_{\m{Z}}\left(\alpha \widetilde{\m{\xi}}\right)\right]\right)+\frac{\lambda}{2}\left(\left\|\widetilde{\mathcal{R}}_{\m{Z}}\left(\alpha\widetilde{\m{\xi}}\right)\right\|_{\text{F}}^2 + \left\|\left(\widetilde{\mathcal{R}}_{\m{Z}}\left(\alpha\widetilde{\m{\xi}}\right)\right)^\dag\right\|_{\text{F}}^2\right)\\
=&\phi\left(\alpha\right)+\frac{\lambda}{2}\left(\left\|\m{Z}+\alpha\widetilde{\m{\xi}}\right\|_{\text{F}}^2+\left\|\left(\m{Z}+\alpha\widetilde{\m{\xi}}\right)^\dag\right\|_{\text{F}}^2\right)\\
=&\phi\left(\alpha\right)+\frac{\lambda}{2}\left(\left\|\m{Z}\right\|_{\text{F}}^2+\alpha\left\langle\m{Z},\widetilde{\m{\xi}}\right\rangle+\alpha^2\left\|\widetilde{\m{\xi}}\right\|_{\text{F}}^2\right)+\frac{\lambda}{2}\left\|\left(\m{Z}+\alpha\widetilde{\m{\xi}}\right)^\dag\right\|_{\text{F}}^2.
\end{aligned}
\end{equation}
Let $\m{Z}=\m{U}_1\m{\Sigma}_1\m{V}_1^H$ and $\widetilde{\m{\xi}}=\m{U}_2\m{\Sigma}_2\m{V}_2^H$ be the SVDs. Define $\begin{bmatrix}\m{U}_1&\m{U}_2\end{bmatrix}=\m{U}_3\m{\Sigma}_3\m{V}_3^H$ and $\begin{bmatrix}\m{V}_1&\m{V}_2\end{bmatrix}=\m{U}_4\m{\Sigma}_4\m{V}_4^H$ as SVDs. Then,
\begin{equation}
\begin{aligned}
\left\|\left(\m{Z}+\alpha\widetilde{\m{\xi}}\right)^\dag\right\|_{\text{F}}^2=&\left\|\left(\m{U}_1\m{\Sigma}_1\m{V}_1^H+\alpha\m{U}_2\m{\Sigma}_2\m{V}_2^H\right)^\dag\right\|_{\text{F}}^2\\
=&\left\|\left(\m{U}_3\m{\Sigma}_3\m{V}_3^H\begin{bmatrix}\m{\Sigma}_1&\\&\alpha\m{\Sigma}_2\end{bmatrix}\left(\m{U}_4\m{\Sigma}_4\m{V}_4^H\right)^H\right)^\dag\right\|_{\text{F}}^2\\
=&\left\|\left(\m{\Sigma}_3\m{V}_3^H\begin{bmatrix}\m{\Sigma}_1&\\&\alpha\m{\Sigma}_2\end{bmatrix}\m{V}_4\m{\Sigma}_4\right)^\dag\right\|_{\text{F}}^2.
\end{aligned}
\end{equation}
Therefore, once $\left\langle\m{Z},\widetilde{\m{\xi}}\right\rangle$, $\left\|\widetilde{\m{\xi}}\right\|_{\text{F}}^2$, and the necessary SVDs are computed (requiring $O\left(K^2N\right)$ flops), the evaluation of $\widehat{h}\left(\left[\widetilde{\mathcal{R}}_{\m{Z}}\left(\alpha \widetilde{\m{\xi}}\right)\right]\right)$ incurs only an additional $O\left(K^3\right)$ flops. Algorithm 4.4 conducts a backtracking line search to identify the smallest integer $q\geq0$ such that $\alpha^{\left(q\right)}=\frac{1}{2^q}\alpha^{\left(0\right)}$ satisfies Armijo condition.
\begin{algorithm}[htbp]
\label{alg-step size}
    \caption{Calculate the step size} 
    {\bf Input:}
    $\displaystyle \m{Z}\in\mathbb{C}^{p\times K}_*$, horizontal vector $\displaystyle \widetilde{\m{\xi}}\in \mathscr{H}_{\m{Z}}$, and parameter $C$.\\
    {\bf Output:} 
    Step size $\displaystyle \alpha^{(q)}$.\\
    \hspace*{0.02in} 1: Compute $\displaystyle \left\langle\m{Z},\widetilde{\m{\xi}}\right\rangle$ and $\displaystyle \left\|\widetilde{\m{\xi}}\right\|_{\text{F}}^2$; \hfill$\triangleright$ $\#$ $O\left(K^2N\right)$flops\\
    \hspace*{0.02in} 2: Compute SVDs: $\displaystyle \m{Z}=\m{U}_1\m{\Sigma}_1\m{V}_1^H$ and $\displaystyle \widetilde{\m{\xi}}=\m{U}_2\m{\Sigma}_2\m{V}_2^H$; \hfill$\triangleright$ $\#$ $O\left(K^2N\right)$flops\\
    \hspace*{0.02in} 3: Compute SVDs: $\displaystyle \begin{bmatrix}\m{U}_1&\m{U}_2\end{bmatrix}=\m{U}_3\m{\Sigma}_3\m{V}_3^H$ and $\displaystyle \begin{bmatrix}\m{V}_1&\m{V}_2\end{bmatrix}=\m{U}_4\m{\Sigma}_4\m{V}_4^H$;\\
\hspace*{0.02in} \hfill$\triangleright$ $\#$ $O\left(K^2N\right)$flops\\
    \hspace*{0.02in} 4: Compute $\displaystyle \widehat{h}\left(\left[\m{Z}\right]\right)$; \\
    \hspace*{0.02in} 5: Compute $\displaystyle \widetilde{g}_{\m{Z}}\left({\text{grad}\,\widehat{h}\left(\left[\m{Z}\right]\right)}_{\uparrow_{\m{Z}}}, \widetilde{\m{\xi}}\right)$; \hfill$\triangleright$ $\#$ $O\left(K^2N\right)$flops\\
    \hspace*{0.02in} 6: \textbf{for} $q = 0,\cdots$ \textbf{do}\\
\hspace*{0.46in} Set $\displaystyle \alpha^{\left(q\right)}=\frac{1}{2^q}\alpha^{\left(0\right)}$; \\
\hspace*{0.46in} Evaluate $\displaystyle \widehat{h}\left(\left[\widetilde{\mathcal{R}}_{\m{Z}}\left(\alpha^{(q)}\widetilde{\m{\xi}}\right)\right]\right)$ from $\displaystyle \left\|\left(\m{\Sigma}_3\m{V}_3^H\begin{bmatrix}\m{\Sigma}_1&\\&\alpha\m{\Sigma}_2\end{bmatrix}\m{V}_4\m{\Sigma}_4\right)^\dag\right\|_{\text{F}}^2$;\\
\hspace*{0.02in}\hfill$\triangleright$ $\#$ $O\left(K^3\right)$flops\\
\hspace*{0.46in} \textbf{if} $\displaystyle \widehat{h}\left(\left[\m{Z}\right]\right) - \widehat{h}\left(\left[\widetilde{\mathcal{R}}_{\m{Z}}\left(\alpha^{(q)}\widetilde{\m{\xi}}\right)\right]\right) \geq -C \alpha^{\left(q\right)} \widetilde{g}_{\m{Z}}\left({\text{grad}\,\widehat{h}\left(\left[\m{Z}\right]\right)}_{\uparrow_{\m{Z}}}, \widetilde{\m{\xi}}\right)$ \textbf{then}\\
\hspace*{0.92in} \textbf{break}.\\
\hspace*{0.46in} \textbf{end if}\\
    \hspace*{0.02in} 7: \textbf{end for}
\end{algorithm}

Combining the two stages, initial estimation and backtracking refinement, the total computational complexity for step size selection is $O\left(K^2N+KN\log N\right)$ flops.

Neglecting lower-order terms, the overall computational complexity of HT-RCGD per iteration is given by
\begin{equation}
\begin{aligned}
O\left(K^2N+KN\log N\right)=O\left(KN\max\left\{K,\log N\right\}\right).
\end{aligned}
\end{equation}
Table \ref{table} summarizes the per-iteration computational complexity of several representative methods for spectral compressed sensing, including enhanced matrix completion (EMaC)~\cite{chen2013spectral}, atomic norm minimization (ANM)~\cite{tang2013compressed}, H-PGD~\cite{cai2018spectral}, SH-PGD~\cite{li2024projected}, Hankel-based Preconditioned Fast Iterative Hard Thresholding (H-PFIHT)~\cite{bian2024preconditioned-SAM}, H-LPPG~\cite{yao2025low}, HT-PGD~\cite{wu2024fast}, and the proposed HT-RCGD. Both EMaC and ANM are considered to be solved using the interior-point method~\cite{ben2001lectures}.

\begin{table}
\renewcommand{\arraystretch}{1.5}
\setlength{\tabcolsep}{24pt}
\begin{center}
	\begin{tabular}{cc}
		\hline
		 Algorithm &  Time Complexity per Iteration\\ 
		\hline
		EMaC&$O\left(N^{5.5}\right)$\\
		\hline
		ANM&$O\left(N^{4.5}\right)$\\
		\hline
		H-PGD&$O\left(KN\max\left\{K,\log N\right\}\right)$\\
		\hline
		SH-PGD&$O\left(KN\max\left\{K,\log N\right\}\right)$\\
		\hline
		H-PFIHT&$O\left(KN\max\left\{K,\log N\right\}\right)$\\
		\hline
		H-LPPG&$O\left(K^4N+K^3N\log N\right)$\\
		\hline
		HT-PGD&$O\left(KN\max\left\{K,\log N\right\}\right)$\\
		\hline
		HT-RCGD&$O\left(KN\max\left\{K,\log N\right\}\right)$\\
		\hline
	\end{tabular}
\end{center}
	\caption{Comparisons of time complexity per iteration.}
\label{table}
\end{table}

\subsection{Convergence}

We establish the convergence of HT-RCGD in the following theorem, which guarantees the existence of at least one accumulation point, and further shows that every accumulation point is a critical point.
\begin{thm}
\label{thm-convergence}
Let the sequence $\left\{\m{Z}_t\right\}$ be generated by HT-RCGD. Then, the following statements hold true: (i) at least one accumulation point exists; and (ii) 
\begin{equation}
\label{convergence}
\begin{aligned}
\lim_{t\to\infty}\widetilde{g}_{\m{Z}_t}\left({\text{grad}\,\widehat{h}\left(\left[\m{Z}_t\right]\right)}_{\uparrow_{\m{Z}_t}},{\text{grad}\,\widehat{h}\left(\left[\m{Z}_t\right]\right)}_{\uparrow_{\m{Z}_t}}\right)=0.
\end{aligned}
\end{equation}
\end{thm}

\begin{sloppypar}
\begin{proof}
The proof proceeds in two parts. First, we establish the existence of at least one accumulation point. Inspired by \cite[Proposition 4.2]{vandereycken2013low}, we aim to show that the iterates remain in a compact subset of the quotient manifold $\mathbb{C}^{p\times K}_*/\mathcal{O}^K$. Consider the sublevel set $\left\{\left[\m{Z}\right] : \widehat{h}\left(\left[\m{Z}\right]\right) \leq \widehat{h}\left(\left[\m{Z}_0\right]\right)\right\}$. Since $h(\cdot)$ is non-negative, it follows that $\left\{\left[\m{Z}\right] : \widehat{h}\left(\left[\m{Z}\right]\right) \leq \widehat{h}\left(\left[\m{Z}_0\right]\right)\right\}
\subseteq
\left\{\left[\m{Z}\right] : \lambda\psi\left(\left[\m{Z}\right]\right) \leq \widehat{h}\left(\left[\m{Z}_0\right]\right)\right\}$. For any representative $\m{Z} \in \left[\m{Z}\right]$ in the latter set, $\frac{\lambda}{2}\left(\left\|\m{Z}\right\|_{\text{F}}^2 + \left\|\m{Z}^\dag\right\|_{\text{F}}^2\right) \leq \widehat{h}\left(\left[\m{Z}_0\right]\right)$, which implies $\left\|\m{Z}\right\|_{\text{F}}^2 \leq \frac{2\widehat{h}\left(\left[\m{Z}_0\right]\right)}{\lambda}$ and $\left\|\m{Z}^\dag\right\|_{\text{F}}^2 \leq \frac{2\widehat{h}\left(\left[\m{Z}_0\right]\right)}{\lambda}$. Thus, the singular values satisfy $\sigma_{\max}(\m{Z}) \leq \sqrt{\frac{2\widehat{h}([\m{Z}_0])}{\lambda}}$ and $\sigma_{\min}(\m{Z}) \geq \sqrt{\frac{\lambda}{2\widehat{h}([\m{Z}_0])}}$, where $\sigma_{\max}(\m{Z})$ and $\sigma_{\min}(\m{Z})$ denote the largest and smallest singular values of $\m{Z}$, respectively. Consequently, the set $\left\{\left[\m{Z}\right] : \sigma_{\max}(\m{Z}) \leq \sqrt{\frac{2\widehat{h}([\m{Z}_0])}{\lambda}}, \;\sigma_{\min}(\m{Z}) \geq \sqrt{\frac{\lambda}{2\widehat{h}([\m{Z}_0])}}\right\}$ is closed and bounded, and hence compact~\cite{vandereycken2013low}. Since the objective value decreases monotonically under the Armijo step size rule, all iterates remain in this compact set, ensuring the existence of an accumulation point.

Second, we prove the asymptotic vanishing of the Riemannian gradient by contradiction. Suppose~\eqref{convergence} does not hold. Then there exist $\varepsilon>0$ and a subsequence $\left\{\m{Z}_t\right\}_{t\in\mathcal{K}}$ such that $\widetilde{g}_{\m{Z}_t}\left({\text{grad}\,\widehat{h}\left(\left[\m{Z}_t\right]\right)}_{\uparrow_{\m{Z}_t}},{\text{grad}\,\widehat{h}\left(\left[\m{Z}_t\right]\right)}_{\uparrow_{\m{Z}_t}}\right)\geq\varepsilon,\;\forall t\in\mathcal{K}$. Let ${\m{Z}}^*$ be an accumulation point of this subsequence. Then $\widetilde{g}_{\m{Z}^*}\left({\text{grad}\,\widehat{h}\left(\left[\m{Z}^*\right]\right)}_{\uparrow_{\m{Z}^*}},{\text{grad}\,\widehat{h}\left(\left[\m{Z}^*\right]\right)}_{\uparrow_{\m{Z}^*}}\right)>0$. However, from~\eqref{update direction}, the update direction $\m{\xi}_t$ satisfies the gradient-related condition. By \cite[Theorem 4.3.1]{absil2008optimization}, every accumulation point is a critical point, contradicting our assumption that~\eqref{convergence} does not hold. Therefore, the Riemannian gradient must vanish asymptotically.

Combining both parts completes the proof.
\end{proof}
\end{sloppypar}


\section{Numerical Simulations}
\label{Sec:simulation}

In this section, we present numerical experiments to evaluate the performance of HT-RCGD for spectral compressed sensing. The initialization procedure for HT-RCGD is as follows. The variable $\widetilde{\m{\xi}}$ is initialized as $\m{0}$. Given the linear observation ${\mathcal{P}_\Omega}\left(\m{y}\right)\in\mathbb{C}^M$, the initialization of $\m{Z}$ follows the approach suggested in~\cite{wu2024fast}. Specifically, we compute the rank-$K$ approximation of $\frac{M}{N}\mathcal{H}{\mathcal{P}_\Omega}^*{\mathcal{P}_\Omega}\left(\m{y}\right)$ and perform its Takagi factorization $\m{U}\m{\Sigma}\m{U}^\top$, where $\m{U}\in\mathbb{C}^{p\times K}$ and $\m{\Sigma}\in\mathbb{R}^{p\times K}$~\cite{chebotarev2014singular}. We then initialize $\m{Z}=\m{U}\m{\Sigma}^{\frac{1}{2}}$. HT-RCGD is terminated either when the norm of the Riemannian gradient $\displaystyle \widetilde{g}_{\m{Z}}\left({\text{grad}\,\widehat{h}\left(\left[\m{Z}\right]\right)}_{\uparrow_{\m{Z}}},{\text{grad}\,\widehat{h}\left(\left[\m{Z}\right]\right)}_{\uparrow_{\m{Z}}}\right)$ falls below a threshold $\epsilon=10^{-6}$, or when the number of iterations reaches $3\times 10^3$. We set the gradient-related parameter $c=10^{-8}$ and the Armijo parameter $C=10^{-5}$.

We compare HT-RCGD with several existing methods, including EMaC~\cite{chen2013spectral}, ANM~\cite{tang2013compressed}, SH-PGD~\cite{li2024projected}, H-PFIHT~\cite{bian2024preconditioned-SAM}, and HT-PGD~\cite{wu2024fast}. For EMaC and ANM, we utilize the interior-point method-based SDPT3 solver~\cite{toh2012implementation} via the CVX toolbox~\cite{grant2014cvx}. Additionally, we introduce the upper bounds on the number of uniquely identifiable spectral components as performance benchmarks. It is shown in~\cite{wax1989unique} that the conditions for deterministic unique identifiability and almost sure identifiability are:
\begin{equation}
\label{upper-bound0}
\begin{aligned}
K&<\frac{M+1}{2},
\end{aligned}
\end{equation}
and
\begin{equation}
\label{upper-bound}
\begin{aligned}
K&<\frac{2M}{3},
\end{aligned}
\end{equation}
respectively. 

Each coefficient $s_k$ is generated independently with amplitude $1+\left|\omega_k\right|$ and random phases, where $\omega_k$ follows the standard Gaussian distribution. To assess the accuracy of spectral-sparse signal reconstruction, we compute the normalized mean square error (NMSE) defined as: $\text{NMSE}=\frac{\left\|\widehat{\m{y}}-\m{y}\right\|_{2}^2}{\left\|\m{y}\right\|_{2}^2}$, where $\widehat{\m{y}}$ and $\m{y}$ represent the estimated and true signals, respectively. In the low-rank Hankel-Toeplitz model, we set $p=\max\left\{\left\lceil\frac{N+1}{2}\right\rceil,K+1\right\}$, where $\left\lceil\cdot\right\rceil$ denotes the ceiling operator. Consequently, the rank-$K$ Hankel-Toeplitz matrix is of size $2p \times 2p$, and the embedded Hankel and Toeplitz submatrices are rank-deficient. In the low-rank Hankel model, we set the Hankel matrix size to $p\times p$ to ensure it is rank-deficient. Similarly, in the low-rank Toeplitz model, the matrix size is set to $\max\left\{N,K+1\right\}\times\max\left\{N,K+1\right\}$ to achieve rank deficiency.

In $\emph{Experiment \ 1}$, we evaluate the numerical performance of HT-RCGD in terms of convergence rate by computing the NMSE at each iteration. The experiment is conducted with $N=70$ and $M=40$. Consider $K=6$ spectral components with frequencies $\m{f}=\left[0.0573,0.1382,0.7245,0.7561,0.8846,0.9954\right]^\top$. The results from a single trial are illustrated in Fig.~\ref{fig.Convergence}. The proposed HT-RCGD has the fastest convergence rate, followed by H-FTIHT, SH-PGD and HT-PGD, demonstrating that the use of the established quotient geometry effectively accelerates convergence.
\begin{figure}
\begin{center}
\includegraphics[width=2.6in]{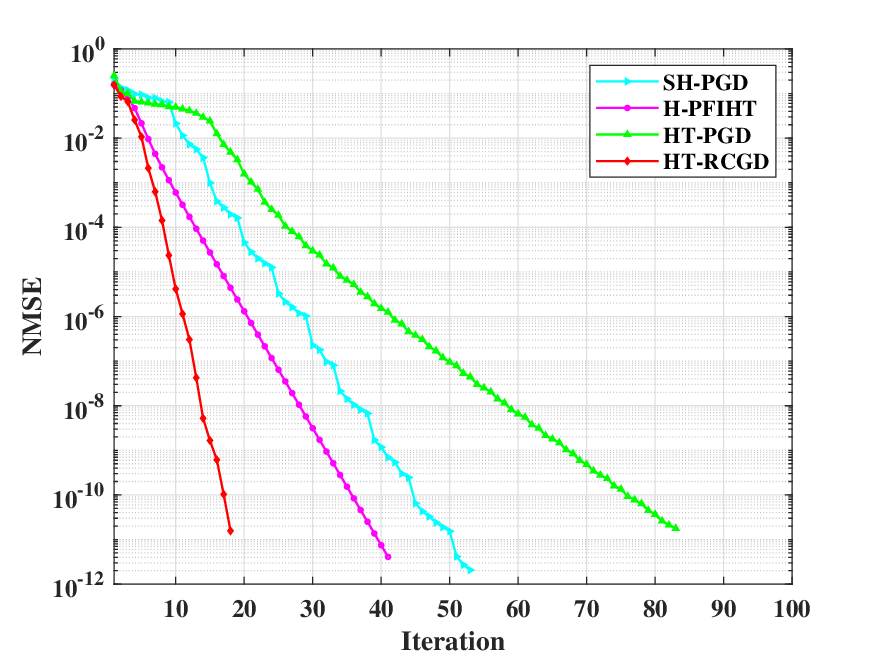}
\caption{NMSE vs iteration.}
\end{center}
\label{fig.Convergence}
\end{figure}

In $\emph{Experiment \ 2}$, we investigate the phase transition behavior and iteration efficiency of HT-RCGD to evaluate its sample complexity performance under a minimum frequency separation constraint. The signal length is fixed at $N=70$, while the number of observed samples $M$ varies in the set $\left\{5,8,\cdots,68\right\}$ and the sparsity level $K\in\left\{1,3,\cdots,37\right\}$. The spectral-sparse signal is considered successfully recovered if $\text{NMSE}\leq10^{-6}$. The success rate is computed by averaging the results over $50$ Monte Carlo trials for each $\left\{M,K\right\}$ pair. The phase transition results when frequencies have a minimum separation of $1.5/N$ are shown in Fig.~\ref{fig.PT,with-sep}. The red solid and dashed lines represent the upper bounds in~\eqref{upper-bound0} and~\eqref{upper-bound}. It is observed that ANM has a smaller complete failure region than EMaC, highlighting the benefit of fully exploiting the signal structure. Partially due to the same reason, HT-PGD and HT-RCGD achieve a larger successful recovery region compared to H-PFIHT and SH-PGD.
\begin{figure}
\begin{center}
\begin{minipage}[b]{.3\linewidth}
\includegraphics[width=1.5in]{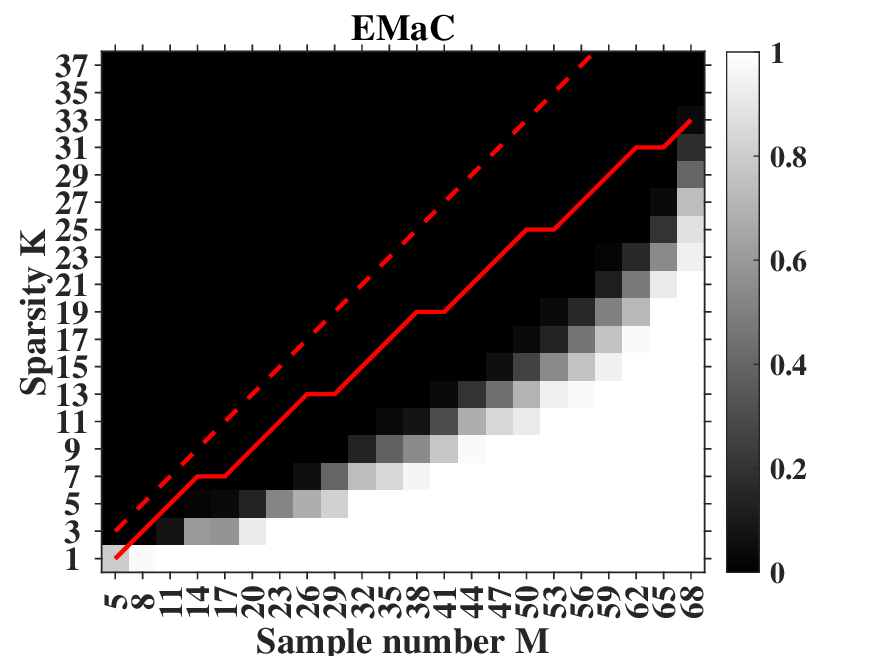}
\includegraphics[width=1.5in]{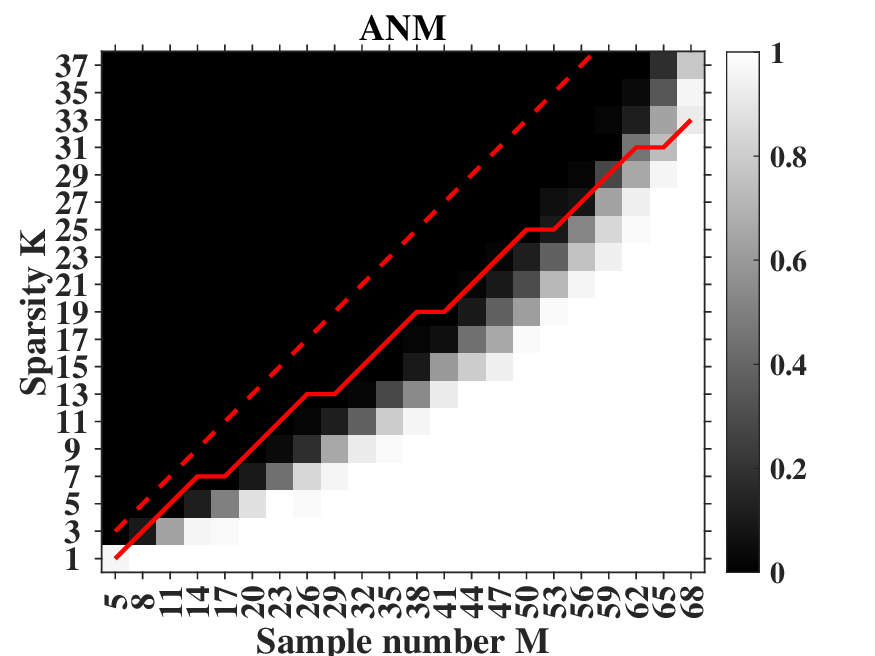}
\end{minipage}
\begin{minipage}[b]{.3\linewidth}
\includegraphics[width=1.5in]{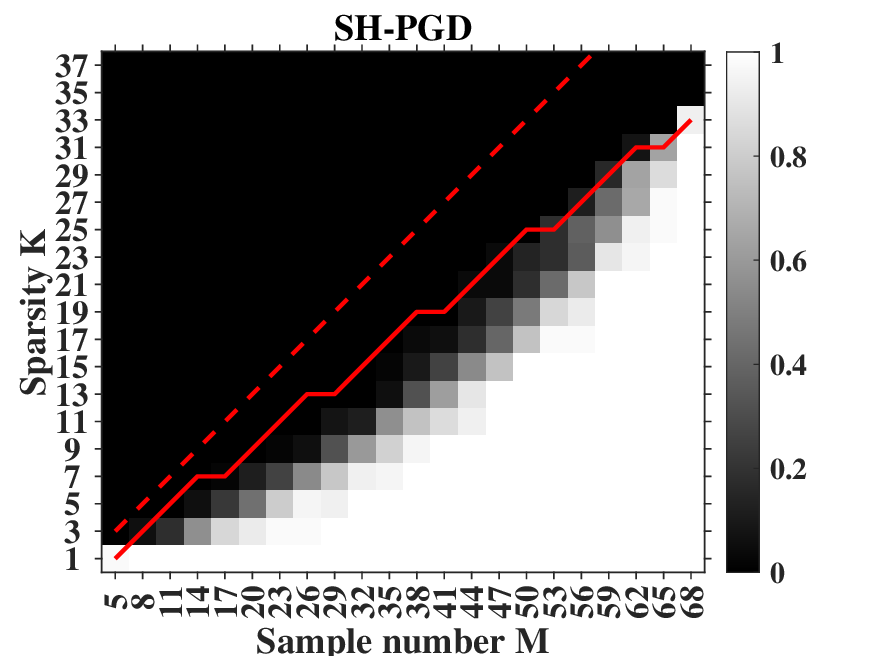}
\includegraphics[width=1.5in]{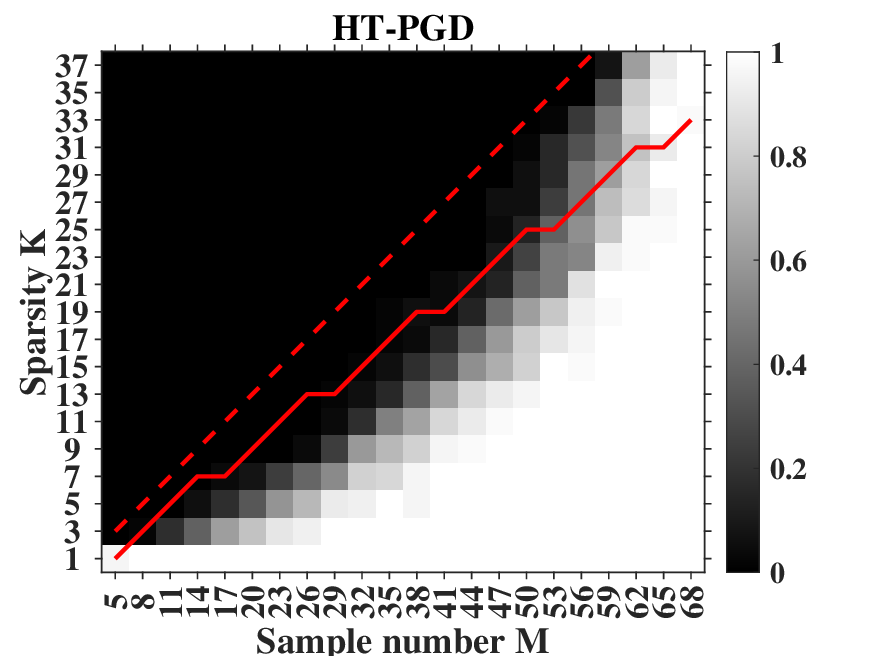}
\end{minipage}
\begin{minipage}[b]{.3\linewidth}
\includegraphics[width=1.5in]{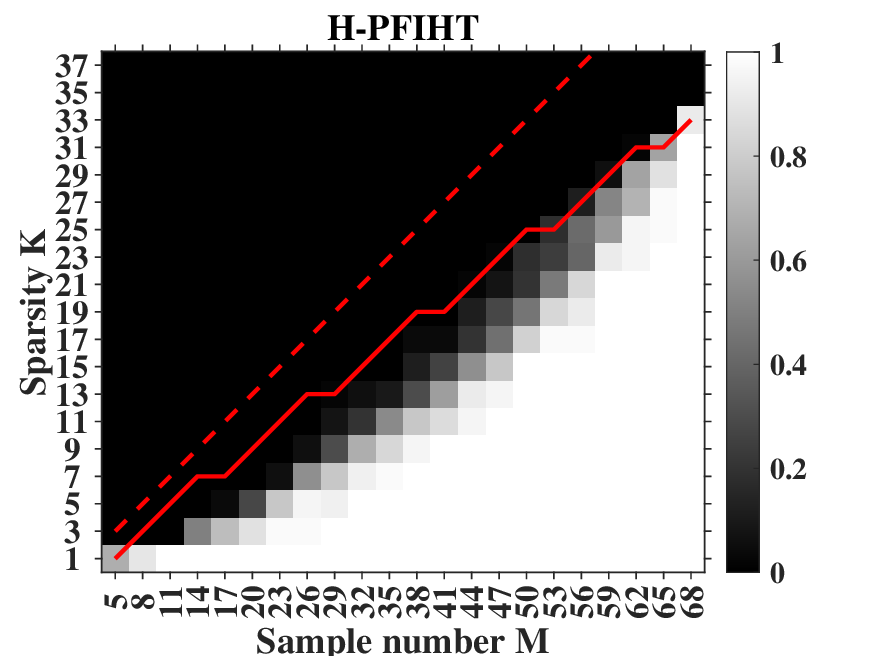}
\includegraphics[width=1.5in]{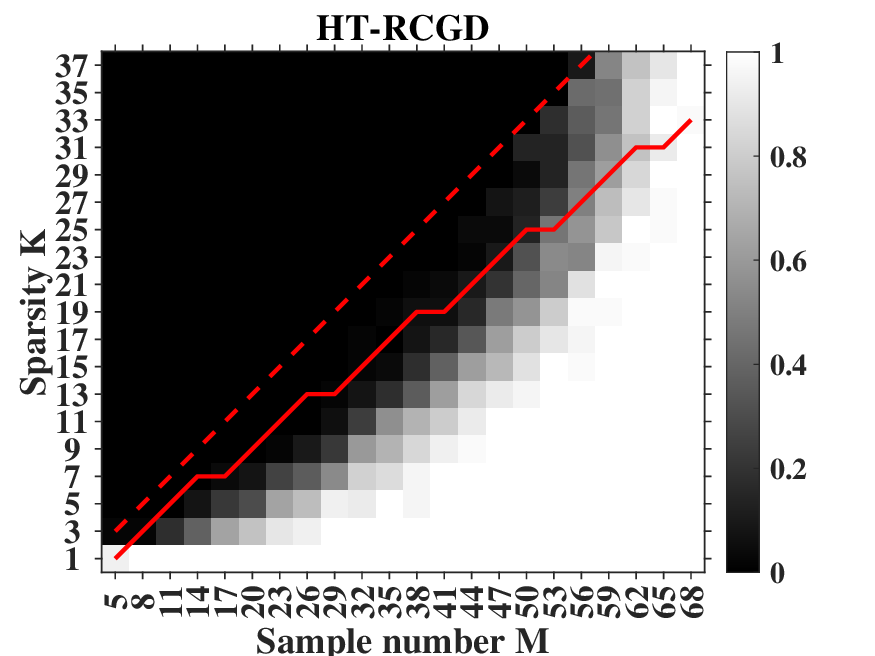}
\end{minipage}
\caption{Phase transition behaviors for frequencies with minimum separation. White means complete
success and black means complete failure.}
\end{center}
\label{fig.PT,with-sep}
\end{figure}
The iteration counts for SH-PGD, H-PFIHT, HT-PGD and HT-RCGD are shown in Fig.~\ref{fig.Iters,with-sep}. When the condition in~\eqref{upper-bound} is not satisfied (i.e., recovery is theoretically impossible), the iteration number is set to the maximum value for all algorithms to indicate unsuccessful recovery. It is observed that HT-RCGD demonstrates better iteration efficiency, especially when both the number of samples and spectral components are large.
\begin{figure}
\begin{center}
\begin{minipage}[b]{.4\linewidth}
\includegraphics[width=1.75in]{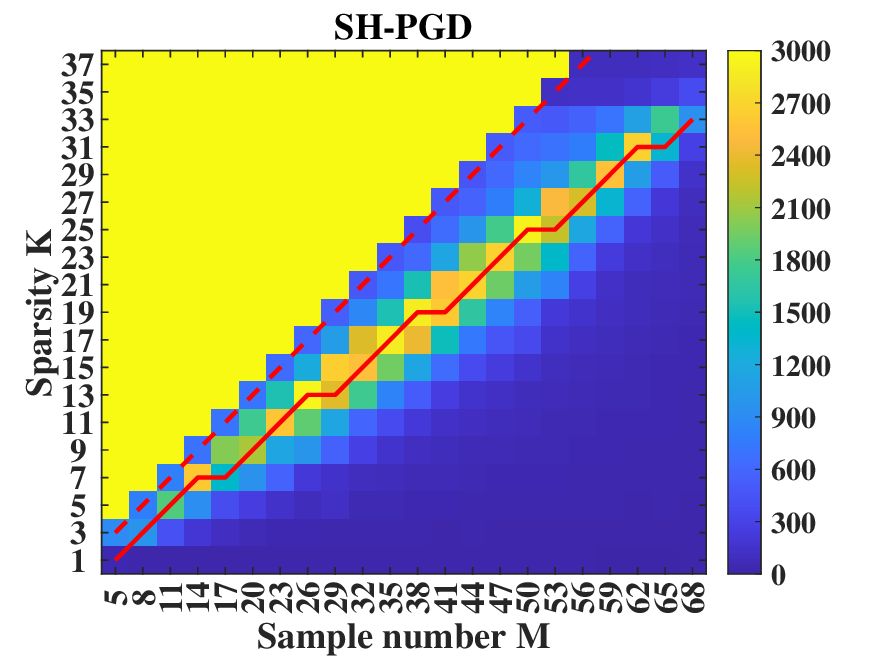}
\includegraphics[width=1.75in]{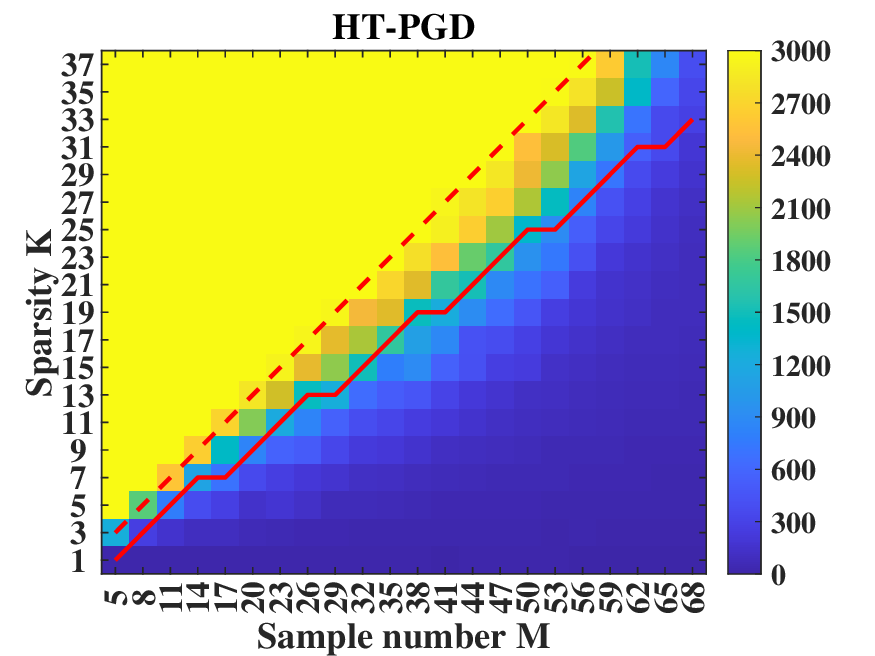}
\end{minipage}
\begin{minipage}[b]{.4\linewidth}
\includegraphics[width=1.75in]{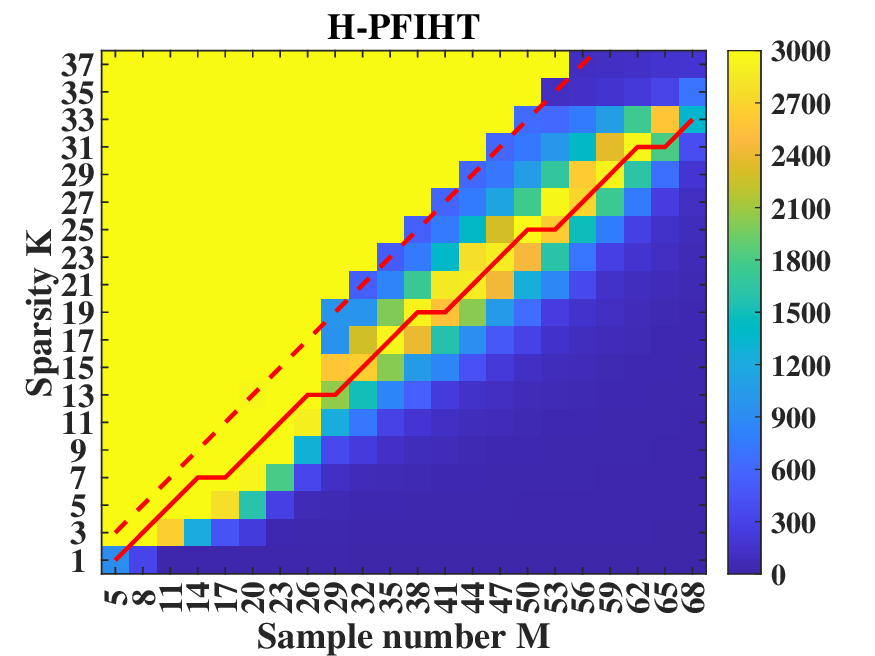}
\includegraphics[width=1.75in]{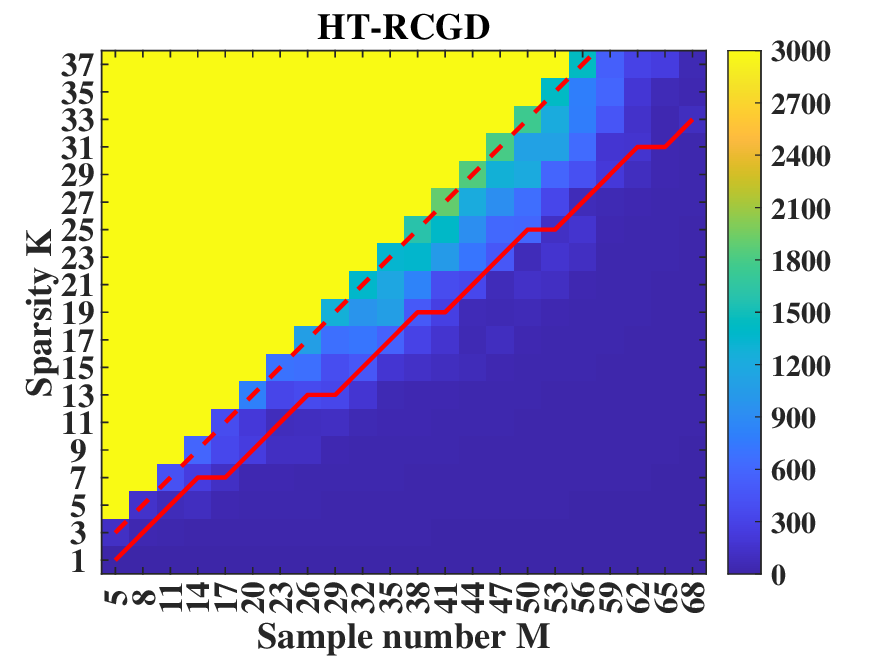}
\end{minipage}
\caption{Comparison of iteration numbers for varying $M$ and $K$ under frequency separation constraint.}
\end{center}
\label{fig.Iters,with-sep}
\end{figure}

In $\emph{Experiment \ 3}$, we evaluate the phase transition behavior and iteration efficiency of HT-RCGD in scenarios where the frequency components are randomly generated and may arbitrarily close to each other. The remaining experimental setup follows that of $\emph{Experiment \ 2}$. The phase transition results are presented in Fig.~\ref{fig.PT,without-sep}. In this setting, EMaC achieves a larger successful recovery region and a narrower phase transition boundary compared to ANM, confirming that the Hankel model solved via convex relaxation is robust to closely spaced frequencies, as previously reported in~\cite{yi2023separation}. Among all methods, HT-RCGD achieves the smallest complete failure region, demonstrating that HT-RCGD requires fewer samples for accurate signal recovery and remains robust even when frequencies are closely spaced.
\begin{figure}
\begin{center}
\begin{minipage}[b]{.3\linewidth}
\includegraphics[width=1.5in]{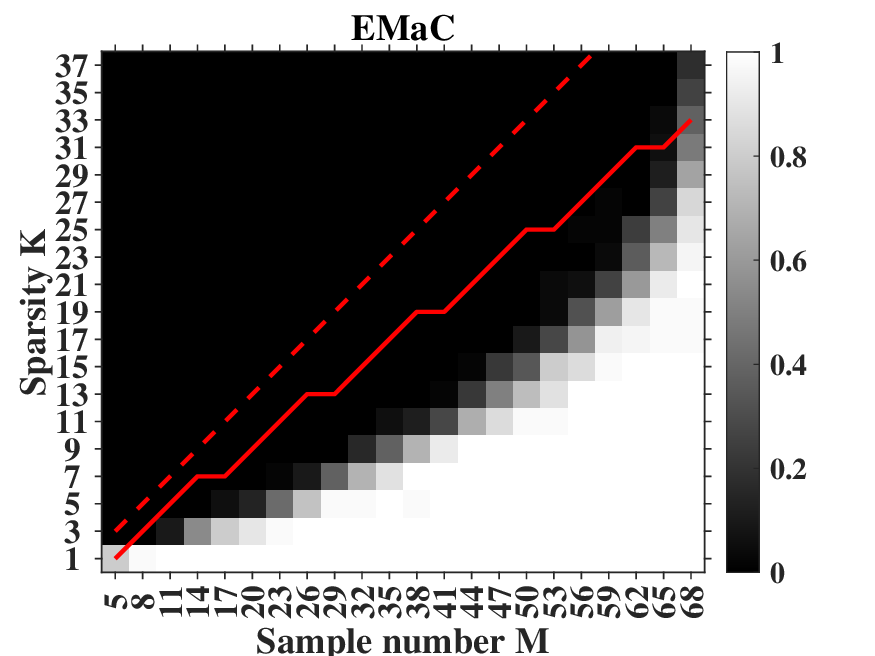}
\includegraphics[width=1.5in]{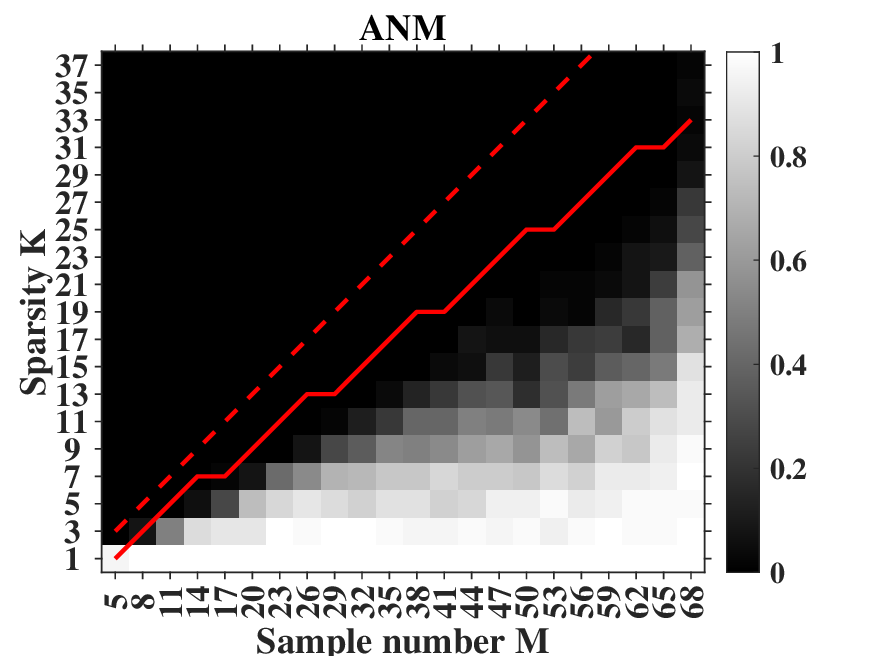}
\end{minipage}
\begin{minipage}[b]{.3\linewidth}
\includegraphics[width=1.5in]{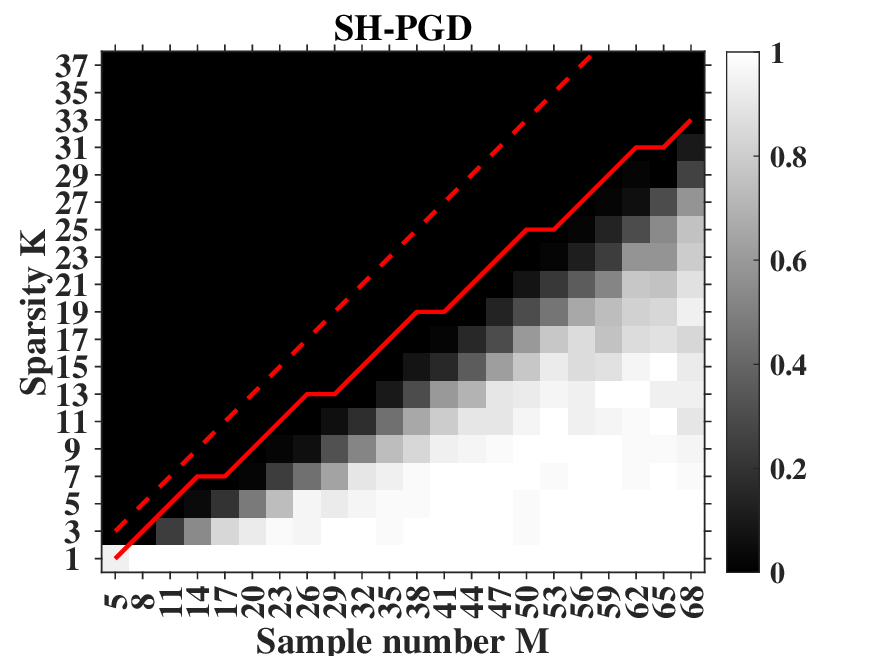}
\includegraphics[width=1.5in]{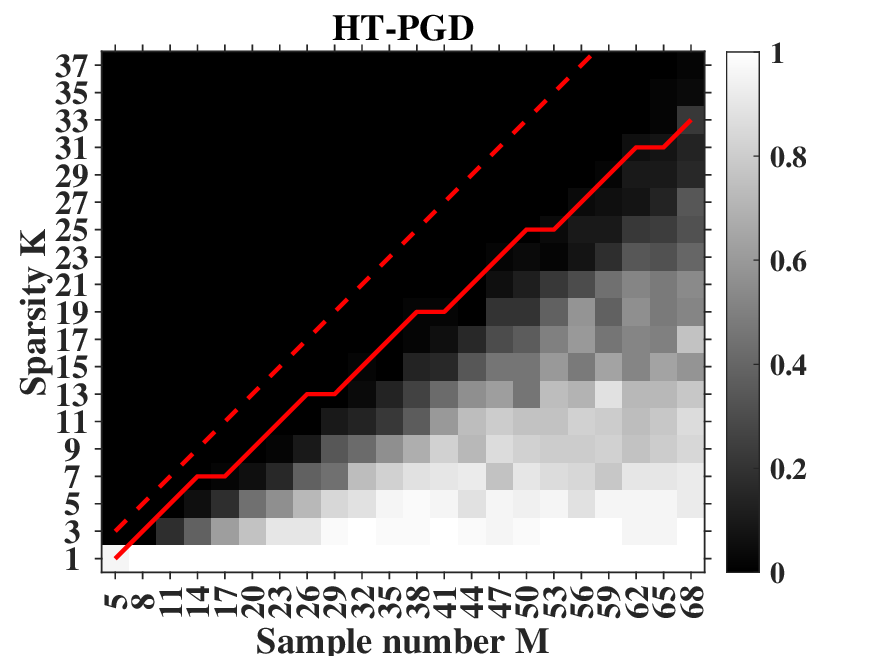}
\end{minipage}
\begin{minipage}[b]{.3\linewidth}
\includegraphics[width=1.5in]{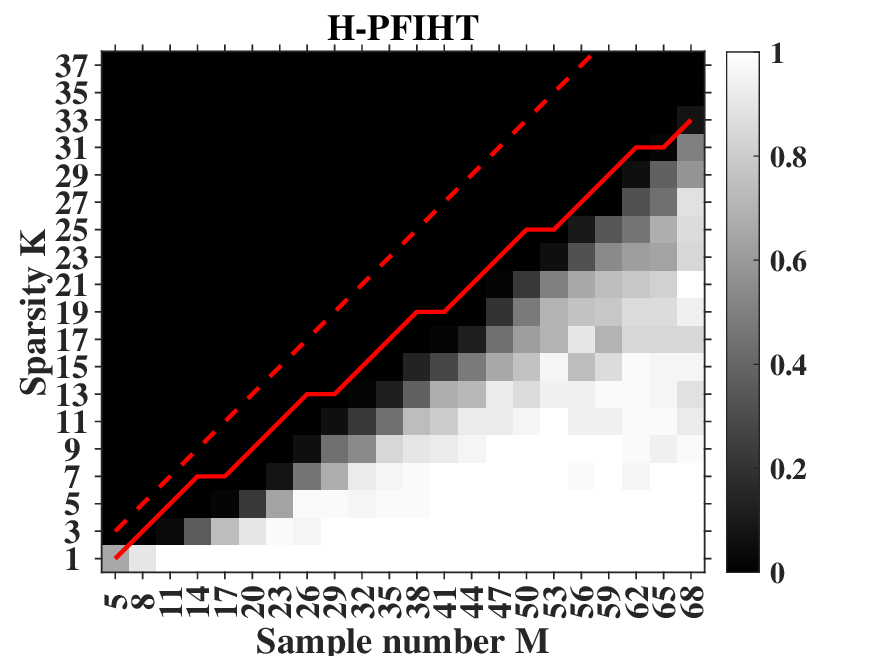}
\includegraphics[width=1.5in]{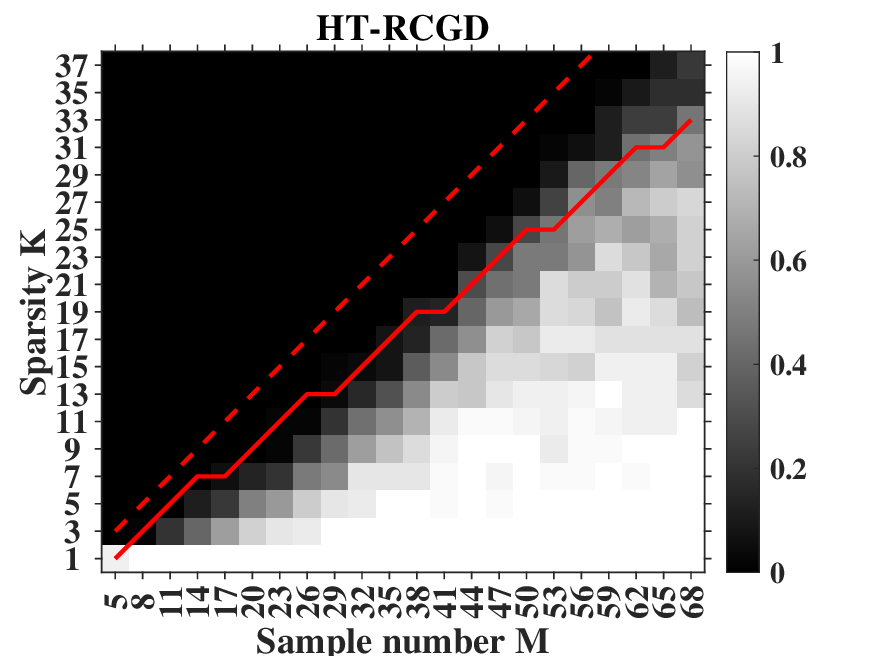}
\end{minipage}
\caption{Phase transition behavior for frequencies without minimum separation. White means complete
success and black means complete failure.}
\end{center}
\label{fig.PT,without-sep}
\end{figure}
The corresponding iteration counts for SH-PGD, H-PFIHT, HT-PGD and HT-RCGD are shown in Fig.~\ref{fig.Iters,without-sep}. HT-RCGD shows significantly better iteration efficiency than SH-PGD and HT-PGD. While it may require more iterations than H-PFIHT when the number of spectral components is small, HT-RCGD tends to require fewer iterations when both the number of samples and spectral components are large.
\begin{figure}
\begin{center}
\begin{minipage}[b]{.4\linewidth}
\includegraphics[width=1.75in]{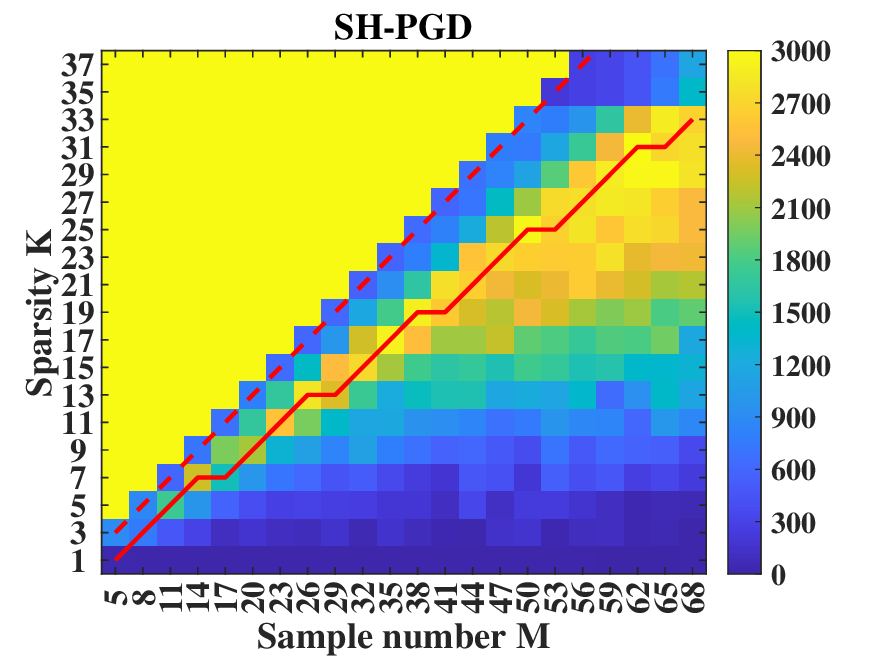}
\includegraphics[width=1.75in]{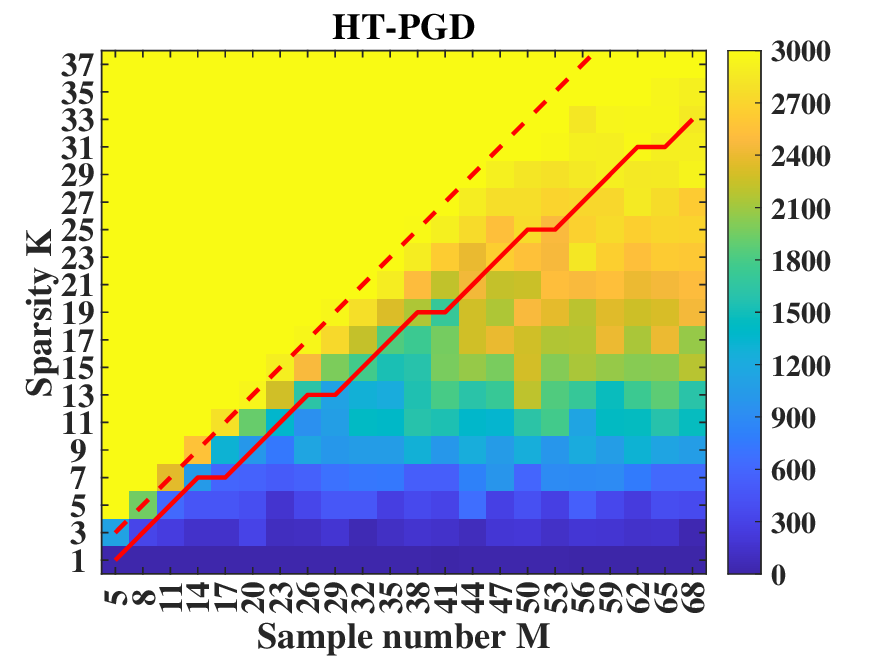}
\end{minipage}
\begin{minipage}[b]{.4\linewidth}
\includegraphics[width=1.75in]{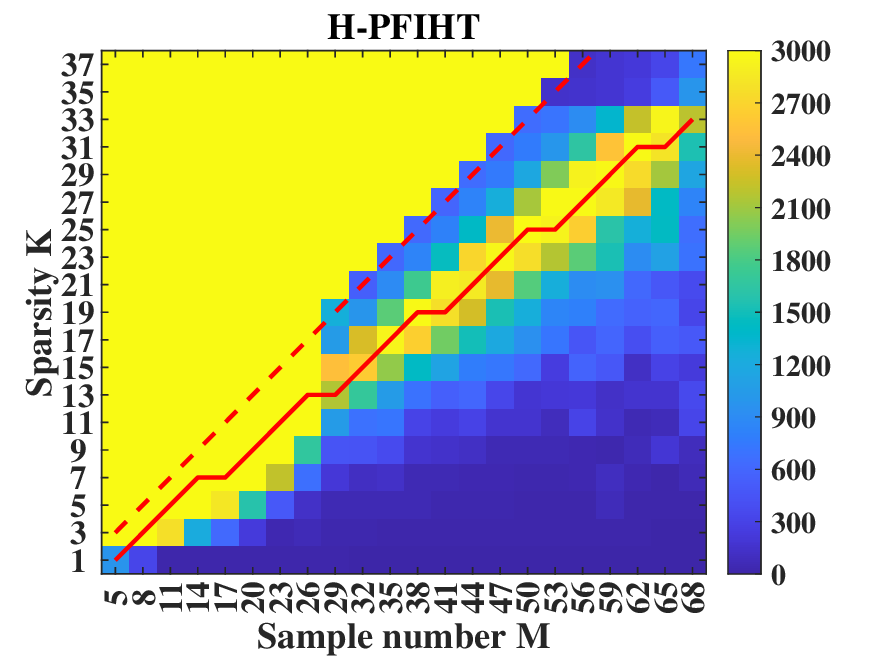}
\includegraphics[width=1.75in]{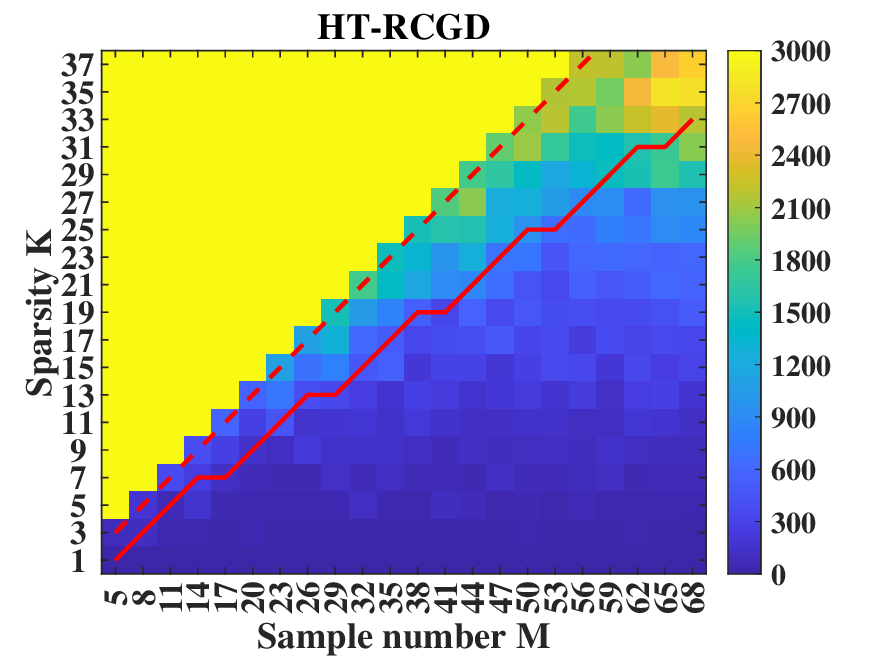}
\end{minipage}
\caption{Comparison of iteration numbers for varying $M$ and $K$ without frequency separation.}
\end{center}
\label{fig.Iters,without-sep}
\end{figure}

In $\emph{Experiment \ 4}$, we assess the computational efficiency of HT-RCGD by measuring its execution time for different signal lengths. We vary the signal length $N$ from $50$ to $20000$ and set the number of observed samples as $M=\left\lfloor0.8N\right\rfloor$, where $\left\lfloor\cdot\right\rfloor$ denotes the floor function (rounding down to the nearest integer). We randomly generate $K=6$ frequencies with a minimum separation of $1/N$. The average computational time, computed over $50$ Monte Carlo trials for successful signal recovery, is plotted in Fig.~\ref{fig.Running Time}. HT-RCGD demonstrates significantly faster computational performance compared to EMaC and ANM, validating the time complexity analysis and highlighting its capability to handle large-scale signal reconstruction problems. While HT-RCGD is slightly slower than SH-PGD and H-PFIHT, as it requires additional matrix computations to calculate the Riemannian conjugate gradient and the optimal stepsize at each iteration, HT-RCGD outperforms HT-PGD when $N\geq500$.
\begin{figure}
\begin{center}
\includegraphics[width=2.6in]{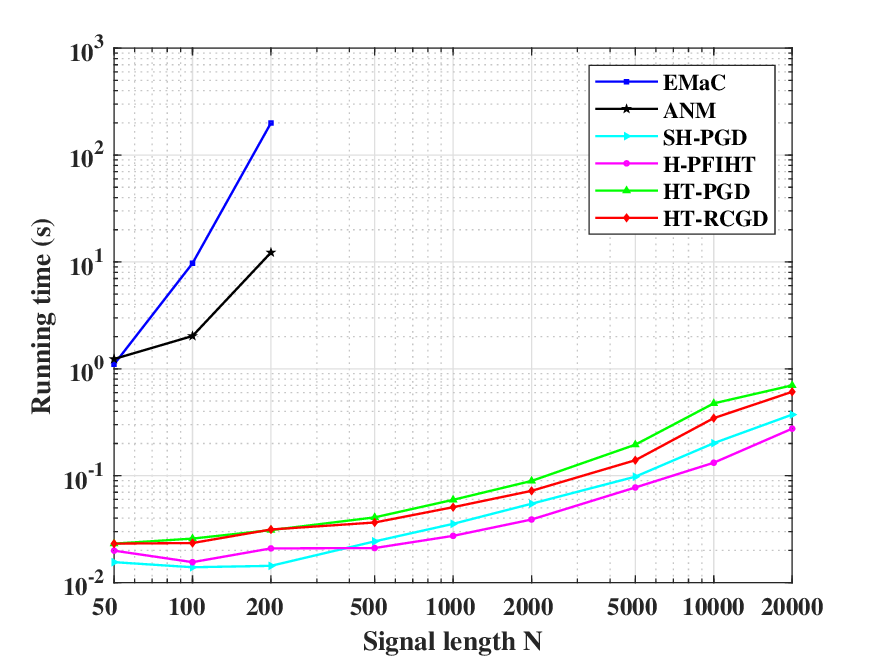}
\caption{Execution time versus signal length.}
\end{center}
\label{fig.Running Time}
\end{figure}

In summary, the numerical results demonstrate that HT-RCGD achieves the best performance in terms of both accuracy and computational efficiency.

\section{Conclusion}
\label{Sec:conclusion}

In this paper, we proposed an equivalence class optimization framework on a quotient manifold for spectral compressed sensing, based on the low-rank Hankel-Toeplitz model. The problem is solved using a Riemannian conjugate gradient method, termed as HT-RCGD. HT-RCGD fully exploits both the structure of undamped spectral-sparse signals and the geometric properties of the quotient manifold. Extensive numerical experiments demonstrate that HT-RCGD outperforms state-of-the-art algorithms in terms of accuracy and convergence rate.

%
%
%

\bibliographystyle{siamplain}
\bibliography{strings}
\end{document}